\documentclass{amsart}
\usepackage{amsfonts}
\usepackage{amssymb}

\usepackage{amscd}
\usepackage{color}

\newtheorem{theorem}{Theorem}[section]

\newtheorem{lemma}[theorem]{Lemma}
\newtheorem{remark}{Remark}[section]
\newtheorem{proposition}{Proposition}[section]
\newtheorem{definition}{Definition}[section]
\numberwithin{equation}{section}

\begin{document}
\title[the NLS equation with a potential in 2D]{ On the scattering problem for the nonlinear Schr\"{o}dinger equation with a potential in 2D}

\author[V. Georgiev]{Vladimir Georgiev}
\address{Department of Mathematics, University of Pisa, Largo Bruno Pontecorvo 5, Pisa, 56127, Italy}
\email{georgiev@dm.unipi.it}

\author[C. Li]{Chunhua Li}
\address{Department of Mathematics, College of Science, Yanbian University, No. 977
Gongyuan Road, Yanji City, Jilin Province,133002, China}
\email{sxlch@ybu.edu.cn}

\maketitle

\begin{abstract}
We consider the scattering problem for the nonlinear Schr\"{o}dinger
equation with a potential in two space dimensions.
Appropriate resolvent estimates
are proved  and applied  to estimate the operator $A(s)$ appearing in commutator relations.
The equivalence between the operators
$\left(-\Delta_{V}\right)^{\frac{s}{2}}$ and $\left(-\Delta \right)^{\frac{s}{2}}$
 in $L^{2}$ norm sense
for $0\leq s <1$ is investigated
by using free resolvent estimates and Gaussian estimates for the heat kernel of the Schr\"{o}dinger operator $-\Delta_{V}$.
Our main result guarantees the global existence of solutions and time
decay of the solutions assuming initial data have small weighted Sobolev norms. Moreover, the global solutions obtained in the main result scatter.
\end{abstract}

\section{Introduction and main results}

We consider the following nonlinear Schr\"{o}dinger equation
\begin{eqnarray}
\left\{
\begin{array}{l}
i \partial_t u + \frac{1}{2}(\Delta-V) u = \lambda |u|^{p-1}u, \\
u(1, x) = u_{0} (x)
\end{array}
\right. \label{NLS}
\end{eqnarray}
in $(t,x) \in \mathbb{R}\times \mathbb{R}^{2}$, where $\Delta$ is the 2-dimensional Laplacian,
 $u=u(t,x)$ is a complex-valued unknown function,
 $t\geq 1$, $p>2, \lambda \in \mathbb{C}\backslash \{0\} $, $\text{Im}\lambda \leq 0$,
$V(x)$ is a  real valued measurable function defined in $\mathbb{R}^{2}$.

In this paper we assume the  time-independent potential $V(x)$  satisfying the following three hypotheses.

\rm {(H1) } The real valued potential $V(x)$ is of the $C^1$ class on $\mathbb{R}^{2}$ and satisfies the  decay estimate $\left|V(x)\right|+\left|x\cdot \nabla V(x)\right|\leq \frac{c}{<x>^\beta}$, where $c>0$ and $\beta>3$;

\rm {(H2) } The potential  $V(x)$ is non-negative;

\rm {(H3) } Zero is a regular point.

We notice that the operator $\Delta_{V}= \Delta -V(x)$ is self-adjoint one by the assumption \rm {(H1)}.
The  assumption \rm {(H2)} and the spectral theorem guarantee that the spectrum of $-\Delta_{V} \subset [0,\infty).$
The short range decay assumption \rm {(H1)} implies that $-\Delta_{V}$ has no positive  eigenvalues due to Agmon's result in \cite{Agmon2}. Combining this fact, the assumption \rm {(H3)}  and Theorem 6.1 in \cite{Agmon}, we see that the spectrum of $-\Delta_{V} \subset [0,\infty)$ is absolutely continuous ( as it was deduced also in \cite{Sch}).

The assumption \rm {(H3)} is not always necessary. We can see in Appendix II that stronger decay of the potential $V(x)$ with $\beta > 10$ in \rm {(H1)} can guarantee that zero is a regular point provided $V \geq 0$ by Theorem 6.2 in \cite{JN01}. Another situation, \rm {(H3) }  and appropriate resolvent estimates are obtained by Theorem 8.2 and Remark 9.2 in \cite{Mo18} under the additional assumption $\partial_r (r V(x)) \leq 0.$

The importance of self-adjointness of quantum Hamiltonians has been
shown, since the work of Von Neumann about 1930  (see \cite{simon17}).
After the Gross-Pitaevskii equation was presented in  1960s,
many crucial problems in quantum mechanics can be reduced to
the study of (\ref{NLS}).
 However, there
is few  research about the asymptotic behavior of solutions for the nonlinear Schr\"{o}dinger
equation (\ref{NLS}) (see \cite{CGV2014}, \cite{GV2012},
\cite{Li2017},  \cite{Mi07}).

In the case of  $V(x)\equiv0$, it is well known that $p=1+\frac{2}{n}$ can be regarded as a boarderline of the short range and long range interactions  to the equation (\ref{NLS}) (see \cite{MS91}, \cite{Str81}, \cite{str74} and \cite{Ba1984}).
The existence of modified wave operators of
the cubic nonlinear Schr\"{o}dinger equation (\ref{NLS}) with $V(x)\equiv0$ and $\lambda \in \mathbb{R}\setminus\{0\}$ in $\mathbb{R}$ was first studied by Ozawa in \cite{Ozawa}.
In the case of the space dimension $n=1, 2, 3$, Hayashi and Naumkin showed the completeness of scattering operators and the decay estimate of the critical nonlinear Schr\"{o}dinger equations (\ref{NLS}) with $V(x)\equiv0$ and $\lambda \in \mathbb{R}\setminus\{0\}$ in \cite{HN98}.
The initial value problem for the critical nonlinear Schr\"{o}dinger equation (\ref{NLS}) with $V(x)\equiv0$
and $\lambda \in \mathbb{R}\setminus\{0\}$ in space dimensions $n \geq 4$ was considered by
Hayashi, Li and Naumkin in \cite{HLN18}. They obtained the two side sharp time decay estimates of solutions in the uniform norm.
There have been some research about
decay estimates of solutions to the subcritical nonlinear Schr\"{o}dinger equation (\ref{NLS}) with $V(x)\equiv0$ and $\text{Im}{\lambda}<0$
for arbitrarily large initial data
(see e.g. \cite{JJL} and \cite{KS}).
Segawa, Sunagawa and Yasuda considered a sharp lower bound for the lifespan of small solutions to the subcritical Schr\"{o}dinger equation  (\ref{NLS}) with $V(x)\equiv0$ and $\text{Im}{\lambda}>0$
in the space dimension $n=1,2,3$ in \cite{SaSuYa}.
For the systems of nonlinear Schr\"{o}dinger equations,
the existence of modified wave operators to a quadratic system in $\mathbb{R}^2$ was studied in \cite{HLN16}, and initial value problem for a cubic derivative system in $\mathbb{R}$ was investigated in \cite{LS16}.

When $V(x)\not\equiv 0$, the existence of wave operators for three dimensional Schr\"{o}dinger operators
with singular potentials was proved by Georgiev and Ivanov in \cite{GI2005}.
Georgiev and Velichkov studied decay estimates for the  nonlinear
Schr\"{o}dinger equation (\ref{NLS})  with  $p>\frac{5}{3}$  in  $\mathbb{R}^{3}$ in \cite{GV2012}.
In \cite{CGV2014},  Cuccagna, Georgiev and Visciglia  considered decay and scattering of small solutions to the nonlinear Schr\"{o}dinger equation  (\ref{NLS})  with $p>3$  in $\mathbb{R}$.
Li and Zhao proved decay and scattering of solutions for the nonlinear
Schr\"{o}dinger equation  (\ref{NLS})  with $1+\frac{2}{n}<p\leq 1+\frac{4}{n-2}$ in \cite{Li2017}, when the
space dimension $n\geq 3$.
$L^p$-boundedness of wave operators for two dimensional Schr\"{o}dinger operators
was first studied by Yajima in \cite{Ya99}.
In \cite{Mi07} Mizumachi studied the asymptotic stability of a small solitary wave to the nonlinear
Schr\"{o}dinger equation (\ref{NLS}) with $p\geq 3$ in $\mathbb{R}^{2}$.
As far as we know, the time decay and scattering problem for the supercritical nonlinear
Schr\"{o}dinger equations (\ref{NLS}) with $p>2$ in  $\mathbb{R}^{2}$
has not been shown.
In this paper, our aim is to study  the time decay and scattering problem
for  (\ref{NLS}) with $V(x)$  under the assumptions $\rm {(H1) }-\rm {(H3) }$
for $p>2$.

We now introduce some notations. $L^{p}(\mathbb{R}^{n})$ denotes usual
Lebesgue space on $\mathbb{R}^{n}$ for $1\leq p \leq \infty$.
For $m,s\in \Bbb{R}$, weighted Sobolev space $
H^{m,s}\left( \Bbb{R}^{n}\right) $ is defined by
\begin{equation}
H^{m,s}\left( \Bbb{R}^{n}\right) =\left\{ f\in \mathcal{S}^{\prime} \left(
\Bbb{R}^{n}\right) ;\left\| f\right\| _{H^{m,s}\left( \Bbb{R}^{n}
\right) }=\left\|(1+|x|^{2})^{\frac{s}{2}}
(I-\Delta)^{\frac{m}{2}} f \right\| _{L^{2}
\left( \Bbb{R}^{n}\right) }<\infty \right\} .
\notag
\end{equation}
We write $H^{m,0}\left( \mathbb{R}^{n}\right) =H^{m}\left( \mathbb{R}^{n}\right)$ for simplicity.
For $s \geq 0,$  the homogeneous Sobolev spaces are denoted by
\begin{equation}
\dot{H}^{s,0}\left( \Bbb{R}^{n}\right) =\left\{ f\in \mathcal{S}^{\prime} \left(
\Bbb{R}^{n}\right) ;\left\| f\right\| _{\dot{H}^{s,0}\left( \Bbb{R}^{n}
\right) }=\left\|(-\Delta)^{\frac{s}{2}} f \right\| _{L^{2}
\left( \Bbb{R}^{n}\right) }<\infty \right\}
\notag
\end{equation}
and
\begin{equation}
\dot{H}^{0,s}\left( \Bbb{R}^{n}\right) =\left\{ f\in \mathcal{S}^{\prime} \left(
\Bbb{R}^{n}\right) ;\left\| f\right\| _{\dot{H}^{0,s}\left( \Bbb{R}^{n}
\right) }=\left\||x|^{s} f \right\| _{L^{2}
\left( \Bbb{R}^{n}\right) }<\infty \right\}.
\notag
\end{equation}
For $1 \leq p \leq \infty$ and $s>0$,
we denote the space $L^{p,s}$ with the norm
\begin{eqnarray}
\|f\|_{L^{p,s}}=\|<x>^{s}f\|_{L^{p}\left(\mathbb{R}^{2}\right)}.\notag
\end{eqnarray}
We define the dilation operator by
\begin{equation}
\left(D_{t}\phi \right)(x)=\frac{1}{(it)^{\frac{n}{2}}}\phi \left(\frac{x}{t}\right) \notag
\end{equation}
for $t \neq 0$ and define $M(t)=e^{-\frac{i}{2t}
|x|^{2}}$ for $t \neq 0$. Evolution operator $U(t)$ is written as
\begin{equation}
U(t)=M(-t)D_{t}\mathcal{F}M(-t), \notag
\end{equation}
where $\mathcal{F}$ and $\mathcal{F}^{-1}$ denote the Fourier transform
and its inverse respectively.
The standard generator of Galilei transformations is given as
\begin{equation}
J(t)=U(t)xU(-t)=x+it\nabla, \notag
\end{equation}
which is also represented as
\begin{equation}
J(t)=M(-t)it \nabla M(t) \notag
\end{equation}
for $t \neq 0$.
Fractional power of $J(t)$ is defined as
\begin{equation}
|J|^{a}(t)=U(t)|x|^{a}U(-t), a>0, \notag
\end{equation}
which is also represented as (see \cite{HO1988})
\begin{equation}
|J|^{a}(t)=M(-t)(-t^{2}\Delta)^{\frac{a}{2}}M(t) \notag
\end{equation}
for $t \neq 0$. Moreover we have commutation relations
with $|J|^{a}$ and $L=i\partial_{t}+\frac{1}{2}\Delta$ such
that $[L,|J|^{a}]=0$.
In what follows, we denote several positive constants by the
same letter $C$, which may vary from one line to another.
If there exists some constant $C>0$ such that $A \leq CF$,
we denote this fact by $``A \lesssim F"$. Similarly,  $``A \sim F"$
means  $``A \lesssim F"$ and $``F  \lesssim A"$. Let $A$ be a
linear operator from Banach space $X$ to Banach space $Y$.
We denote the operator norm of $A$ by $\|A\|_{X\rightarrow Y}$.

Our main theorem is stated as follows:
\begin{theorem} \label{main theorem}
Assume  that  $V(x)$ satisfies $\rm {(H1) }-\rm {(H3) }$. Let $p>2$. Then there exist
constants $\epsilon_{0}>0$ and $C_{0}>0$ such that  for  any $\epsilon \in (0,\epsilon_{0})$
and  $\|u_{0}\|_{H^{\alpha}(\mathbb{R}^2)\cap \dot{H}^{0,\alpha}(\mathbb{R}^2)} \leq \epsilon,$ where
$1<\alpha<2$,
the  solution $u$ to  (\ref{NLS}) satisfies the time decay estimates
\begin{eqnarray}
\|u\|_{L^{\infty}\left(\mathbb{R}^{2}\right)}\leq C_{0}t^{-1} \epsilon \label{timedecay}
\end{eqnarray}
for $t\geq1$.
Moreover there exists $u_{+}\in L^{2}(\mathbb{R}^{2})$
such that
\begin{eqnarray}
\lim_{t\rightarrow \infty}\|u(t)-e^{i\frac{t}{2}\Delta_V}u_{+}\|_{L^{2}\left(\mathbb{R}^{2}\right)}=0. \label{scattering}
\end{eqnarray}
\end{theorem}

To prove Theorem \ref{main theorem}, we introduce the operators $|J_{V}|^{s}$ and $A(s)$ derived from some commutation relations. The properties of operators  $|J_{V}|^{s}$ and $A(s)$ are shown in Section \ref{Operators}. We present Strichartz estimates by Proposition \ref{proposition 2.2} and Proposition \ref{theorem 2.2} in Section \ref{Strichartz Estimates}. We have
\begin{eqnarray}
\label{c2}
\|A(s)u\|_{L^{q}(\mathbb{R}^{2})} \lesssim \|u\|_{L^{q^{\prime}}(\mathbb{R}^{2})}
\end {eqnarray}
for $1\leq q\leq 2$ and $ 1<s<2$ by using resolvent estimates (Lemma \ref{Lemma Resolvent Estimate 2 }) in Section \ref{The estimates of  $A(s)$}, where $\frac{1}{q}+\frac{1}{q^{\prime}}=1$.
Then we show
\begin{eqnarray}
\label{c1}
\left\| \left(-\Delta_{V}\right)^{\frac{s}{2}}f\right\|_{L^{2}(\mathbb{R}^{2})} \sim
 \left\| \left(-\Delta\right)^{\frac{s}{2}} f\right\|_{L^{2}(\mathbb{R}^{2})}
\end{eqnarray}
for all $0 \leq s< 1$ (see Lemma \ref{equi lemma})
and
\begin{eqnarray} \label{c3}
\left\|\left(-\Delta_{V}\right)^{\frac{s}{2}}f-\left(-\Delta \right)^{\frac{s}{2}}f\right\|_{L^{2}(\mathbb{R}^{2})}\lesssim
\|(-\Delta)^{\frac{s}{2}}f\|_{L^{2}(\mathbb{R}^{2})}^{\frac{\sigma }{s}}\|f\|_{L^{2}(\mathbb{R}^{2})}^{1-\frac{\sigma }{s}},
\end{eqnarray}
for all $1 \leq s< 2$ and $0<\sigma<1$ (see (\ref{n.m0}) in Lemma \ref{Lemma Jv})
in Section \ref{Equivalence}.
We prove our main theorem by using Strichartz estimates, (\ref{c2})
and (\ref{c3}) in Section \ref{Proof}. In Section \ref{Appendix}, we give the proofs
of properties of operators  $|J_{V}|^{s}$ and $A(s)$. We show that zero is not resonance
in Section \ref{resonance}.



\section{Operators $|J_{V}|^{s}$ and $A(s)$}
\label{Operators}
We will introduce the operators $|J_{V}|^{s}$ and $A(s)$
 to consider appropriate Sobolev norms and to
study the asymptotic behavior of solutions to  the equation (\ref{NLS}).

Setting $M(t)=e^{-\frac{i}{2t}|x|^{2}},$ we may  define $|J_{V}|^{s}(t):=M(-t)\left(-t^{2}\Delta_{V}\right)^{\frac{s}{2}}M(t)$.  We shall use the standard notation  $[B, D]=BD-DB$ for the commutator of two operators $B$ and $D$. The key commutator properties of the operator $|J_{V}|^{s}(t)$ are given in  the
following two propositions.
\begin{proposition}\label{proposition 1.1}
Let $A(s):=s\left(-\Delta_{V}\right)^{\frac{s}{2}}+\left[x\cdot\nabla, \left(-\Delta_{V}\right)^{\frac{s}{2}}\right].$
For $s>0,$ we have
\begin{equation}
\left[ i\partial_{t}+\frac{1}{2}\Delta_{V},|J_{V}|^{s}(t)\right]=it^{s-1} M(-t)A(s)M(t) \label{1.11}
\end{equation}
in two space dimensions.
\end{proposition}

We also have
\begin{proposition}\label{proposition 1.01}
Let $W:=2V+x\cdot \nabla V$. For $0<s<2,$ we obtain
\begin{eqnarray}
A(s)=c(s)\int_{0}^{\infty}\tau^{\frac{s}{2}}\left(\tau-\Delta_{V}\right)^{-1}W
\left(\tau-\Delta_{V}\right)^{-1}d\tau  \label{4.3.3}
\end{eqnarray}
in two space dimensions, where $c(s)^{-1}=\int_{0}^{\infty} \tau^{\frac{s}{2}-1}(\tau+1)^{-1}d\tau$.
\end{proposition}
Proposition \ref{proposition 1.1} and  Proposition \ref{proposition 1.01}  are well-known from \cite{CGV2014}
for the case of one-dimensional Schr\"{o}dinger equation (\ref{NLS}) with a potential.
For the convenience of readers, we give proofs of these propositions in the appendix I of this paper.



\section{Strichartz Estimates}
\label{Strichartz Estimates}
Strichartz estimates are important tools  to  investigate  asymptotic behavior of solutions
to some evolution equations, such as  Schr\"{o}dinger equations and wave equations.
The well known homogeneous Strichartz estimate
\begin{eqnarray}
\left\|e^{\frac{i}{2}t\Delta}f\right\|_{L^{p_{2}}(\mathbb{R};L^{q_{2}}(\mathbb{R}^{n}))}\lesssim\|f\|_{L^{2}(\mathbb{R}^{n})} \notag
\end{eqnarray}
and  inhomogeneous  Strichartz estimate
\begin{eqnarray}
\left\|\int_{s<t} e^{\frac{i}{2}(t-s)\Delta}F(s, \cdot)ds\right\|_{L^{p_{2}}(\mathbb{R};L^{q_{2}}(\mathbb{R}^{n}))}\lesssim \|F\|_{L^{p_{1}^{\prime}}(\mathbb{R};L^{q_{1}^{\prime}}(\mathbb{R}^{n}))} \notag
\end{eqnarray}
hold for $n\geq 2$, $f \in  L^{2}(\mathbb{R}^{n})$, and $F \in L^{p_{1}^{\prime}}(\mathbb{R};L^{q_{1}^{\prime}}(\mathbb{R}^{n}))$
if $\frac{2}{p_{j}}+\frac{n}{q_{j}}=\frac{n}{2}, 2\leq p_{j} \leq \infty,2\leq q_{j} \leq \frac{2n}{n-2}, q_{j} \neq \infty,$
$p_{j}^{\prime},q_{j}^{\prime}$ are the dual exponents of $p_{j}$ and $q_{j}, j=1, 2$ (see e.g. \cite{KeTao98}).
We note that both endpoints $(p_{j},q_{j})=(\infty,2)$ and $(p_{j},q_{j})=(2, \frac{2n}{n-2})$ are included in the situation of $n \geq 3$, and only the endpoint $(p_{j},q_{j})=(\infty,2)$ is included in the case of $n=2$ for $j=1, 2$.

In recent years, a large number of works on Strichartz estimates for  Schr\"{o}dinger equations
with potentials $V(x)$ have been investigated (see e.g. \cite{BPSTZ2004},
 \cite{BPSTZ2003}, \cite{GV2012}, \cite{AF2008},  \cite{AFVV}, \cite{Mi07}, \cite{Mo18}, \cite{ste}).
 However, the study of Strichartz estimates for 2d Schr\"{o}dinger equations is essentially restricted to
 the cases of smallness of the magnetic potential and electric potential (see \cite{ste}),
 smallness of the magnetic potential while the electric potential can be large (see \cite{AF2008}), very fast decay of the potential and assumption that zero is a regular point (see \cite{Mi07}), or $V \geq 0$ and $\partial_r (rV) \leq 0$ (see \cite{Mo18}).
In \cite{BPSTZ2003}, Strichartz estimates for Schr\"{o}dinger equations
with the inverse-square potential $\frac{a}{|x|^2}$  in two space dimensions were considered
by Burq, Planchon, Stalker and Tahvildar-Zadeh, where $a$ is a real number.
In \cite{Mi07}, Mizumachi presented Strichartz estimates by the $L^{\infty}-L^{1}$ estimates in \cite{Sch}.
To state the dispersive estimate in \cite{Sch}, we recall the notion zero is a regular point as follow:
\begin{definition}\rm{(see \cite{Sch})}
\label{def}
Let $V\not\equiv 0$ and set $U=signV, v=|V|^{\frac{1}{2}}$. Let $P_v$ be the orthogonal projection
onto $v$ and set $Q=I-P_v$. And let
\begin{equation*}
(G_0 f)(x):=-\frac{1}{2\pi}\int_{\mathbb{R}^2}\log|x-y|f(y)dy.
\end{equation*}
We say that zero is a regular point of the spectrum of $-\Delta_V$, provided $Q(U+vG_0v)Q$ is
invertible on $QL^2(\mathbb{R}^2)$.
\end{definition}
We have
\begin{proposition}
\label{Sch}
\rm{(Dispersive Estimate in \cite{Sch})}
Let $V:\mathbb{R}^2\rightarrow \mathbb{R}$ be a measurable function such that $|V(x)|\leq C(1+|x|)^{-\beta},
\beta>3.$ Assume in addition that zero is a regular point of the spectrum of  $-\Delta_V$. Then we have
\begin{equation*}
\|e^{-\frac{i}{2}t\Delta_V}P_{ac}(H)f\|_{L^{\infty}(\mathbb{R}^2)}\lesssim|t|^{-1}\|f\|_{L^{1}(\mathbb{R}^2)}
\end{equation*}
for all $f \in L^{1}(\mathbb{R}^2)$.
\end{proposition}
The requirement that zero is a regular point is the analogue of the usual condition that zero is neither an eigenvalue nor a resonance (generalized eigenvalue) of $-\Delta_V$.
Under the assumptions of Proposition \ref{Sch}, the spectrum of $-\Delta_V$ on $[0, \infty)$
is purely absolutely continuous, and that the spectrum is pure point on $(-\infty, 0)$ with at
most finitely many eigenvalues of finite multiplicities (See \cite{Sch}). Moreover, any point on the real line different from zero is not a resonance due to the results in \cite{GV2007}.
Therefore, unique candidate for resonant point is the origin and the assumption zero is regular means that zero is not resonance too.

Next, we need the definition of admissible couples appearing in the Schrichartz estimates. The couple $(p,q)$ of positive numbers $p \geq 2, q \geq2 $ is called Schr\"{o}dinger admissible
if it satisfies
\begin{eqnarray}
\frac{1}{p}+\frac{1}{q}=\frac{1}{2},  (p,q)\neq(2,\infty).  \label{admissible pair}
\end{eqnarray}

We have the following
homogeneous  Strichartz estimate by Proposition \ref{Sch}, Theorem 6.1 in \cite{Agmon}, and the methods in \cite{KeTao98}.
We omit the proof.

\begin{proposition} {\label {proposition 2.2}}
\rm{(Homogeneous  Strichartz Estimate)}
Let $(p,q)$ be a Schr\"{o}dinger admissible pair. If $\rm{(H1)} - \rm{(H3)}$ are satisfied,
then we obtain
\begin{eqnarray}
\left\|e^{\frac{i}{2}t\Delta_{V}}f\right\|_{L^{p}(\mathbb{R};L^{q}(\mathbb{R}^{2}))}\lesssim\|f\|_{L^{2}(\mathbb{R}^{2})}
\end{eqnarray}
holds for all $f \in  L^{2}(\mathbb{R}^{2})$.
\end{proposition}
By using Proposition \ref{proposition 2.2} and a result of
Christ-Kiselev lemma (Lemma A.1 in \cite{BM16}), we have the
following result. We skip the proof here.

\begin{proposition} {\label {theorem 2.2}}
\rm{(Inhomogeneous  Strichartz Estimate)}
Let $a, b \in \mathbb{R}$ and  let $(p_{j},q_{j})$ be  Schr\"{o}dinger admissible pairs for $j=1, 2$. Assume $V(x)$ satisfy the  hypotheses $\rm{(H1)}-\rm{(H3)}$.
Then we have
\begin{eqnarray}
\left\|\int_{a}^{t}e^{\frac{i}{2}(t-s)\Delta_{V}}F(s,\cdot)ds\right\|_{L^{p_{2}}( [a, b];L^{q_{2}}(\mathbb{R}^{2}))}\lesssim\|F\|_{L^{p_{1}^{\prime}}([a, b];L^{q_{1}^{\prime}}(\mathbb{R}^{2}))}, \label{3.1}\\
\forall F\in L_{loc}^{1}([a, b],L^{2}(\mathbb{R}^{2}))\cap L^{p_{1}^{\prime}}([a, b],L^{q_{1}^{\prime}}(\mathbb{R}^{2})),  \notag
\end{eqnarray}
where $p_{j}^{\prime},q_{j}^{\prime}$ are the dual exponents of $p_{j}$ and $q_{j}, j=1, 2$.
\end{proposition}
\section{The estimates of  $A(s)$}
\label{The estimates of  $A(s)$}
To derive estimates of  $A(s)$,
 we use free resolvent estimates, following the approach of \cite{Li2017}.
\begin{lemma} \rm{(Free Resolvent Estimates)}
\label{Lemma Resolvent Estimate 2 }
\begin{description}
  \item[i)] For any $ 1< q <\infty, \ \ 0 < s_0 \leq 1,$ one can find $C  = C(q,s_0) > 0$ so that for any $\tau >0$ we have
  \begin{equation} \label{ReEs2}
 \left\|\left(\tau-\Delta\right)^{-1}f\right\|_{L^{q}(\mathbb{R}^{2})}
\leq C \tau^{-s_0} \|f\|_{L^{k}(\mathbb{R}^{2})}, \ \  \frac{1}{k}=\frac{1}{q}+1-s_0;
\end{equation}
  \item[ii)] For any  $$ 1< q <\infty, \ \ 0 < s_0 \leq 1, \ \  a > 2(1-s_0), $$ one can find $C=C(q,s_0,a)>0$ so that  for any $\tau >0$ we have
  \begin{equation} \label{ReEs21} \left\|\left(\tau-\Delta\right)^{-1}<x>^{-a}f\right\|_{L^{q}(\mathbb{R}^{2})}
\leq C \tau^{-s_0} \|f\|_{L^{q}(\mathbb{R}^{2})};
\end{equation}
  \item[iii)] For any  $$ 1< q <\infty, \ \ 0 < s_0 \leq 1, \ \  a > 2(1-s_0), $$ one can find $C=C(q,s_0,a)>0$ so that  for any $\tau >0$ we have
  \begin{equation}\label{ReEs22}
\left\|<x>^{-a}\left(\tau-\Delta\right)^{-1}f\right\|_{L^{q}(\mathbb{R}^{2})}
\leq C \tau^{-s_0} \|f\|_{L^{q}(\mathbb{R}^{2})}.
\end{equation}
\end{description}
\end{lemma}
\begin{proof} To prove (\ref{ReEs2}) we take advantage of the fact that the  Green function
$$ G(x-y;\tau) = \left(\tau-\Delta\right)^{-1}(x-y) $$
of the operator $\left(\tau-\Delta\right)^{-1}$ can be computed explicitly, indeed we have
$$  G(x;\tau) = (2\pi)^{-2}\int_{\mathbb{R}^2} e^{-\mathrm{i} x\xi} \frac{d\xi}{\tau+|\xi|^2} = (2\pi)^{-1} K_{0}(\sqrt{\tau}|x|), $$
where $K_0(r)$ is the modified Bessel function of order $0.$ We have the following  estimates of $K_0(r),$
\begin{equation}\label{eq.be1}
    |K_0(r)| \lesssim
\left\{
  \begin{array}{ll}
    |\ln r |, & \hbox{if $0 \leq r \leq 1$;} \\
    e^{-r}/ \sqrt{r}, & \hbox{if $r>1.$}
  \end{array}
\right.
\end{equation}
This estimate implies $K_0(|x|) \in L^m(\mathbb{R}^2)$ for any $m \in [1,\infty).$ In this way we deduce
$$ \int_{\mathbb{R}^2} \left| G(x;\tau) \right|^m d x = (2\pi)^{-m}  \int_{\mathbb{R}^2} \left| K_{0}(\sqrt{\tau}|x|) \right|^m d x $$
$$\hspace{3.2cm} = \frac{(2\pi)^{-m}}{\tau} \int_{\mathbb{R}^2} \left| K_{0}(|y|) \right|^m d y = \frac{c_1}{\tau},$$
where $y=\sqrt{\tau}|x|$ and we can write
\begin{equation}\label{eq.be2}
   \left\|K_0(\sqrt{\tau} \ \cdot)\right\|_{L^{m}(\mathbb{R}^{2})} = \frac{c_2}{\tau^{1/m}}, \ \forall m \in [1,\infty).
\end{equation}
Applying the Young inequality
$$ \left\|\left(\tau-\Delta\right)^{-1}f\right\|_{L^{q}(\mathbb{R}^{2})} = \left\|K_0(\sqrt{\tau}\  \cdot \ )*f\right\|_{L^{q}(\mathbb{R}^{2})} $$
$$\hspace{4.5cm} \leq \left\|K_0(\sqrt{\tau}\  \cdot \ )\right\|_{L^{m}(\mathbb{R}^{2})}\|f\|_{L^{k}(\mathbb{R}^{2})},$$
where $1+1/q=1/m+1/k,$ combining with \eqref{eq.be2} and choosing $s_0 = 1/m$, we deduce \eqref{ReEs2}. So the assertion i) is verified.

To get \eqref{ReEs21}, we apply the estimate  \eqref{ReEs2}
$$ \left\|\left(\tau-\Delta\right)^{-1}<x>^{-a}f\right\|_{L^{q}(\mathbb{R}^{2})} \lesssim \tau^{-s_0} \|<x>^{-a} f\|_{L^{k}(\mathbb{R}^{2})}, \ 1+1/q=s_0+1/k,$$
and via the H\"older inequality
$$ \|<x>^{-a} f\|_{L^{k}(\mathbb{R}^{2})} \lesssim \| f\|_{L^{q}(\mathbb{R}^{2})}, a> 2 \left( \frac{1}{k}- \frac{1}q \right) $$
and we arrive at \eqref{ReEs21}.

Finally (\ref{ReEs22}) follows from \eqref{ReEs21} by a duality argument.

This completes the proof.
\end{proof}

\begin{remark} The estimates
(\ref{ReEs2}), (\ref{ReEs21})  are not valid for $q=\infty,$ $s_0=0,$ $k=1,$ but they are valid for
$$ 1 \leq  q \leq \infty, \ \ 0 < s_0 \leq 1, \ \ \frac{1}{k}=\frac{1}{q}+1-s_0.$$
In particular,  they are true for
$q=k=1,s_0=1.$

Further, (\ref{ReEs22}) holds when $1 < q \leq \infty$ as in Lemma \ref{Lemma Resolvent Estimate 2 }, but also in the following case
$$ q=1, \ 0 < s_0 \leq 1, \  a >2(1-s_0).$$
\end{remark}

In Proposition \ref{proposition 1.01}, for $0<s<2$ we have
\begin{eqnarray}
A(s)=c(s)\int_{0}^{\infty}\tau^{\frac{s}{2}}\left(\tau-\Delta_{V}\right)^{-1}W
\left(\tau-\Delta_{V}\right)^{-1}d\tau,
\end{eqnarray}
where $W=2V+x\cdot \nabla V$.

Let $G(t,x,y) = e^{t\Delta_V }(x,y) $ be the heat kernel of the Schr\"{o}dinger operator
$-\Delta_{V},$ i.e. it solves
\begin{eqnarray}
\left\{
\begin{array}{l}
\partial_{t}G=(\Delta-V) G, \notag \\
G(0, x,y) = \delta (x-y), \notag
\end{array}
\right.
\end{eqnarray}
where $y\in \mathbb{R}^{2}$. Similarly, $$e^{t\Delta}(x, y)= c_3 t^{-1} \exp \left\{-\frac{ |x-y|^{2}}{4t}\right\}$$ is the heat kernel of $-\Delta,$
so that
$$ e^{ \alpha t \Delta}(x, y) = c_4 t^{-1} \exp \left\{-\frac{|x-y|^{2}}{4\alpha t}\right\}, \ \ \forall \alpha >0.$$
Since we consider the case $V(x) \geq 0,$ one can use Feynman-Kac formula and the the results in  \cite{Si05} deduce the
heat kernel estimate
\begin{eqnarray} \label{eq.hk1}
0 \leq e^{t \Delta_V}(x,y) \lesssim t^{-1} \exp \left\{-\frac{|x-y|^{2}}{4\beta t}\right\},
\end{eqnarray}
where $\beta>0$.
Using \eqref{eq.hk1},  we get the following estimate
\begin{lemma}
\label{lemma 4.1}
Assume that  the hypotheses $\rm {(H1) }$ and $\rm {(H2) }$ are satisfied. Then there exists positive $\beta$ such that
\begin{eqnarray}
 0 \leq e^{t\Delta_V }(x,y) \lesssim  e^{\beta t \Delta}(x, y).  \label{4.13}
\end{eqnarray}
\end{lemma}
Without assumption $V \geq 0$ one can use the estimates from
 \cite{AM1998} and deduce only the estimate
 \begin{eqnarray}
\left|e^{t\Delta_V}(x,y)\right| \lesssim  e^{\gamma t \Delta}(x, y) e^{\omega t},
\end{eqnarray}
where $\gamma>0$ and $\omega$ depends on $V(x)$. This is not sufficient for our goal to control the solution to (\ref{NLS}) with a potential $V(x)$.

Further,
we get the following Lemma
\begin{lemma}
\label{Lemma A(s)}
Let $V(x)$ satisfy $\rm{(H1)}$ and $\rm{(H2)}$.
We have
\begin{eqnarray}
\left\|A(s)f\right\|_{L^{q}(\mathbb{R}^{2})} \lesssim \left\|f\right\|_{L^{q^{\prime}}(\mathbb{R}^{2})}\label{A(s)1}
\end{eqnarray}
for  $1\leq q \leq 2$ and $1< s<2$, where $\frac{1}{q}+\frac{1}{q^{\prime}}=1$.
\end{lemma}
\begin{proof}
By  $(\tau+a)^{-1}=\int_{0}^{\infty}e^{-(a+\tau)t}dt,$ we obtain
\begin{eqnarray}
\left(\tau-\Delta_{V}\right)^{-1}f_\pm&=&\int_{0}^{\infty}e^{-\left(\tau-\Delta_{V}\right)t}f_\pm dt \notag \\
&=&\int_{0}^{\infty}e^{-\tau t}e^{t\Delta_{V}}f_\pm dt.  \label{4.23}
\end{eqnarray}
Since $e^{t\Delta_V}(x, y)$ is the heat kernel of $-\Delta_V$,
 then we have
\begin{eqnarray}
e^{t\Delta_{V}}f_\pm(x)=\int_{\mathbb{R}^{2}} e^{t\Delta_{V}}(x,y)f_\pm(y)dy. \label{add2}
\end{eqnarray}
By the estimate (\ref{4.13}) in Lemma \ref{lemma 4.1} and (\ref{add2}), from (\ref{4.23}) there exist positive $\beta$
such that
\begin{eqnarray}
\left| \left(\tau-\Delta_{V}\right)^{-1}f_\pm \right| & \lesssim &\int_{0}^{\infty}e^{-\tau t}
e^{\beta t \Delta}f_\pm dt \notag \\
&\lesssim & \left(\frac{\tau}{\beta}-\Delta\right)^{-1}f_\pm. \label{4.25}
\end{eqnarray}

Given any $q \in [1,2]$ we can apply
 Proposition \ref{proposition 1.01}, (\ref{4.25}), and via  the H\"older inequality to get
 \begin{eqnarray}
\left\|A(s)f\right\|_{L^{q}(\mathbb{R}^{2})}&\lesssim& \int_{0}^{\infty}\tau^{\frac{s}{2}}
\left\|\left(\tau-\Delta_V \right)^{-1}W
\left(\tau-\Delta_{V}\right)^{-1}f\right\|_{L^{q}(\mathbb{R}^{2})}d\tau \notag\\
&\lesssim &\left(\|f_+\|_{L^{q^{\prime}}(\mathbb{R}^{2})}+\|f_-\|_{L^{q^{\prime}}(\mathbb{R}^{2})}\right)\|W\|_{L^{\omega(q)}(\mathbb{R}^{2})} \notag \\
&& \times \int_{1}^{\infty}\tau^{\frac{s}{2}}
\left\|\left(\frac{\tau}{\beta}-\Delta\right)^{-1}\right\|_{L^{q}(\mathbb{R}^{2})\rightarrow
                               L^{q}(\mathbb{R}^{2}) } \notag \\
&&\qquad \times \left\|\left(\frac{\tau}{\beta}-\Delta\right)^{-1}\right\|_{L^{q^{\prime}}(\mathbb{R}
                               ^{2})\rightarrow L^{q^{\prime}}(\mathbb{R}^{2})}d\tau \notag \\
&&+ \left(\|f_+\|_{L^{q^{\prime}}(\mathbb{R}^{2})}+\|f_-\|_{L^{q^{\prime}}(\mathbb{R}^{2})}\right)\|<x>^{2a}W\|_{L^{\omega(q)}(\mathbb{R}^{2})} \notag \\
&& \times \int_{0}^{1}\tau^{\frac{s}{2}}
\left\|\left(\frac{\tau}{\beta}-\Delta\right)^{-1}<x>^{-a}\right\|_{L^{q}(\mathbb{R}^{2})\rightarrow
                               L^{q}(\mathbb{R}^{2}) } \notag \\
&&\qquad \times \left\|<x>^{-a}\left(\frac{\tau}{\beta}-\Delta\right)^{-1}\right\|_{L^{q^{\prime}}(\mathbb{R}
                               ^{2})\rightarrow L^{q^{\prime}}(\mathbb{R}^{2})}d\tau \notag,
                               \end{eqnarray}
where $ \omega=\omega(q)$ is determined by $\frac{1}{\omega}= \frac{2}{q}-1$ and $a = a(q) $ is an appropriate parameter to be chosen so that we can  apply  Lemma \ref{Lemma Resolvent Estimate 2 } (with $s_0=1$ in (\ref{ReEs2}), $s_1=3/4$ in (\ref{ReEs21}) and (\ref{ReEs22})), i.e.  we have to require
\begin{equation}\label{eq.pr1}
 \frac{1}{2} <  a(q) < 1 + \frac{\beta}2 - \frac{2}q,
\end{equation}
where $\beta>3$ is from our assumption $\mathrm{(H1)}$.
Then we can write
\begin{eqnarray}
\left\|A(s)f\right\|_{L^{q}(\mathbb{R}^{2})}
&\lesssim&  \|f\|_{L^{q^{\prime}}(\mathbb{R}^{2})}\|W\|_{L^{\omega}(\mathbb{R}^{2})}  \int_{1}^{\infty}\tau^{\frac{s}{2}-2}d\tau \notag \\
&&\qquad+   \|f\|_{L^{q^{\prime}}(\mathbb{R}^{2})}\|<x>^{2a}W\|_{L^{\omega}(\mathbb{R}^{2})}  \int_{0}^{1}\tau^{\frac{s}{2}-\frac{3}{2}}d\tau \notag \\
& \lesssim &  \|f\|_{L^{q^{\prime}}(\mathbb{R}^{2})}\|W\|_{L^{\omega}(\mathbb{R}^{2})}+ \|f\|_{L^{q^{\prime}}(\mathbb{R}^{2})}\|<x>^{2a}W\|_{L^{\omega}(\mathbb{R}^{2})}. \notag
\end{eqnarray}
Note that our choice \eqref{eq.pr1} guarantees that
we have
\begin{eqnarray}
&& \left\|<x>^{2a}W\right\|_{L^{\omega}(\mathbb{R}^{2})}
\lesssim\|V\|_{L^{\omega,2a}\left(\mathbb{R}^{2}\right)}+\left\|\sum_{j=1}^{2}x_{j}\frac{\partial V}{\partial x_{j}}\right\|_{L^{\omega, 2a}\left(\mathbb{R}^{2}\right)}\leq C.    \label{4.31}
\end{eqnarray}
So the assertion is proved.
\end{proof}
\begin{remark}
\label{remark A(s)}
By using the similar method as Lemma \ref{Lemma A(s)}, we also have
\begin{eqnarray}
\left\|A(s)f\right\|_{L^{q}(\mathbb{R}^{2})}\lesssim \left\|f\right\|_{L^{q}(\mathbb{R}^{2})} \label{add}
\end{eqnarray}
for  $1\leq q \leq \infty$ and $1<s<2$.
\end{remark}
\section {Equivalence of $\left(-\Delta_{V}\right)^{\frac{s}{2}}$ and $\left(-\Delta \right)^{\frac{s}{2}}$
in $L^{2}(\mathbb{R}^{2})$ norm sense}
\label{Equivalence}

To estimate $\left\| |J_{V}|^{s} (|u|^{p-1}u)\right\|_{ L^{2}(\mathbb{R}^{2})}$
which will be mentioned below,
we study the operator $\left(-\Delta_{V}\right)^{\frac{s}{2}}$ via
heat kernels of some Sch\"{o}dinger operators $-\Delta_{V}$ and $-\beta \Delta$
on $\mathbb{R}^{2}$,
where $\beta>0$.
By Lemma \ref{lemma 4.1}, we obtain the following lemma.
\begin{lemma}
\label{lemma 4.2}
 Assume that  the hypotheses $\rm {(H1) }$ and $\rm {(H2) }$ are satisfied. For $s\geq 0$,
we have
\begin{eqnarray}
\left\| \left(-\Delta\right)^{\frac{s}{2}} f\right\|_{L^{2}(\mathbb{R}^{2})} \lesssim
 \left\| \left(-\Delta_{V}\right)^{\frac{s}{2}}f\right\|_{L^{2}(\mathbb{R}^{2})}. \label{4.9}
\end{eqnarray}
\end{lemma}
\begin{proof}
Obviously (\ref{4.9}) holds in the case of $s=0$.
We focus our attention to the situation of $s>0$.
We show that
\begin{eqnarray}
\left\| \left(-\Delta_{V}\right)^{-\frac{s}{2}}
 \left(-\Delta\right)^{\frac{s}{2}}f\right\|_{L^{2}(\mathbb{R}^{2})} \lesssim
 \left\|  f \right\|_{L^{2}(\mathbb{R}^{2})} \label{04.10}
\end{eqnarray}
for $s>0$.\\

Using
\begin{eqnarray}
a^{-\frac{s}{2}}&=&\frac{1}{\Gamma (\frac{s}{2})}\int_{0}^{\infty}e^{-at}t^{\frac{s}{2}-1}dt\notag
\end{eqnarray}
for $s>0$,  we have
\begin{eqnarray}
\left(\left(-\Delta_{V}\right)^{-\frac{s}{2}}g\right)(x)=\frac{1}{\Gamma (\frac{s}{2})}\int_{0}^{\infty}e^{\Delta_{V}t}g(x)t^{\frac{s}{2}-1}dt.\label{04.12}
\end{eqnarray}
Since $\rm {(H1) }$ and $\rm {(H2) }$ say $-\Delta_{V}$ is a positive self-adjoint operator on $L^{2}(\mathbb{R}^{2})$,
then we have for every $t>0$, $e^{\Delta_{V}t}$ has a jointly continuous
integral kernel $e^{\Delta_{V}t}(x, y)$.
Thus we have
\begin{eqnarray}
e^{\Delta_{V}t}g(x)=\int_{\mathbb{R}^{2}}e^{\Delta_{V}t}(x, y) g(y)dy.\notag
\end{eqnarray}
By the estimate (\ref{4.13}) in Lemma \ref{lemma 4.1}, we have

\begin{eqnarray}
0\leq e^{\Delta_{V}t}g(x)&\lesssim& \int_{\mathbb{R}^{2}}
e^{\beta t \Delta}(x,y)g(y)dy \notag \\
&=&e^{\beta t \Delta}g(x) \label{add4.7}
\end{eqnarray}
for $g\geq0$, where $e^{\beta t \Delta}(x,y)$ is the heat kernel of the Sch\"{o}dinger operator
$-\beta\Delta, \beta>0$.
Then we have from (\ref{04.12}) and (\ref{add4.7})
\begin{eqnarray}
\left(-\Delta_{V}\right)^{-\frac{s}{2}}g(x)
\lesssim \left(-\Delta \right)^{-\frac{s}{2}}g(x)\notag
\end{eqnarray}
for $s>0$ and $g\geq0$.
Thus we have
\begin{eqnarray}
\left(-\Delta_{V}\right)^{-\frac{s}{2}}\left(-\Delta \right)^{\frac{s}{2}}f(x)
\lesssim f(x)\notag
\end{eqnarray}
for $s>0$ and $f\geq0$, where $f=\left(-\Delta\right)^{-\frac{s}{2}}g$ .
Let $B_s=\left(-\Delta_{V}\right)^{-\frac{s}{2}}\left(-\Delta \right)^{\frac{s}{2}}$.
Decomposing $f = f_+-f_-, g=g_+-g_-$  we can deduce
\begin{eqnarray*}
|<B_sf, g>|&=&|<B_s(f_{+}-f_{-}), g_{+}-g_{-}>|\\
&\leq&|<B_sf_{+}, g_{+}>|+|<B_sf_{+}, g_{-}>|\\
\qquad &+&|<B_sf_{-}, g_{+}>|+|<B_sf_{-}, g_{-}>|\\
&\lesssim& \|f\|_{L^{2}(\mathbb{R}^2)}\|g\|_{L^{2}(\mathbb{R}^2)}
\end{eqnarray*}
for $s>0$, without requiring $f \geq 0, g \geq 0.$ The inequality
\begin{equation*}
    \|B_sf\|_{L^2} \lesssim\|f\|_{L^{2}(\mathbb{R}^2)}
\end{equation*}
holds for $s>0$.
Therefore we have the estimate (\ref{04.10}).\\
\end{proof}
\begin{remark}
It is difficult to obtain Gaussian estimates for heat kernel of the Sch\"{o}dinger operator
$-\Delta+V(x)$, especially for $V \leq 0$. There are some sharp Gaussian estimates
 for heat kernel of the Sch\"{o}dinger operator $-\Delta+V(x)$ with $V \geq 0 $
(see e.g. \cite{AM1998}, \cite{BDS2016} and \cite{Zhang2001}).
Especially, the sharp Gaussian Estimates for heat kernel of the Sch\"{o}dinger operator
$-\Delta+V(x)$ with nontrivial $V\geq 0$ in $\mathbb{R}^{2}$ or
 $\mathbb{R}^{1}$ fail (see \cite{BDS2016} ).
 \end{remark}


\begin{lemma}
\label{Lemma Jv0}
Let $V(x)$ satisfy $(\rm{H1})$ and $(\rm{H2})$. For any $0<s<2$ and for any $ q> 2$, we have
\begin{eqnarray} \label{n.m0}
\left\|\left(-\Delta_{V}\right)^{\frac{s}{2}}f-\left(-\Delta \right)^{\frac{s}{2}}f\right\|_{L^{2}(\mathbb{R}^{2})}\lesssim
\|f\|_{L^{q}(\mathbb{R}^{2})}.
\end{eqnarray}
\end{lemma}

\begin{proof}
Since  $(\tau+a)^{-1}=\int_{0}^{\infty}e^{-(a+\tau)t}dt,$ we have
\begin{eqnarray*}
\left(\tau-\Delta_{V}\right)^{-1}g(x)=\int_{0}^{\infty}e^{-\tau t}e^{\Delta_{V}t}g(x) dt.
\end{eqnarray*}
Let $e^{t\Delta_V}(x,y)$ be the heat kernel of the Sch\"{o}dinger operator
$-\Delta_{V}$. Then we have
\begin{eqnarray*}
e^{\Delta_{V}t}g(x)=\int_{\mathbb{R}^{2}}e^{t\Delta_V}(x,y)g(y)dy.
\end{eqnarray*}

By Lemma \ref{lemma 4.1}, we  have positive $\beta$ such that
\begin{eqnarray}
\label{a1}
\left(\tau-\Delta_{V}\right)^{-1}g(x)&\lesssim&\int_{0}^{\infty}e^{-\tau t}
e^{\beta \Delta t}g(x) dt \notag \\
&\lesssim&\left(\frac{\tau}{\beta}-\Delta\right)^{-1}g(x)
\end{eqnarray}
for $g\geq 0$.
Then we have
\begin{eqnarray}
\|\left(\tau-\Delta_{V}\right)^{-1}g\|_{L^2(\mathbb{R}^{2})}\lesssim\left\|\left(\frac{\tau}{\beta}-\Delta\right)^{-1}g\right\|_{L^2(\mathbb{R}^{2})}  \label{a1}
\end{eqnarray}
without requiring $g\geq 0$.

Now we can use the relation
\begin{eqnarray}
\left(-\Delta_{V}\right)^{\frac{s}{2}}f=c(s)(-\Delta_{V})\int_{0}^{\infty}\tau^{\frac{s}{2}-1}
\left(\tau-\Delta_{V}\right)^{-1}fd\tau \notag
\end{eqnarray}
 with
\begin{eqnarray}
c(s)^{-1}=\int_{0}^{\infty}\tau^{\frac{s}{2}-1}(\tau+1)^{-1}d\tau  \notag
\end{eqnarray}
for $0<s<2$. Therefore,
we have
\begin{eqnarray}
\left(-\Delta_{V}\right)^{\frac{s}{2}}f&=&c(s)(-\Delta_{V})\int_{0}^{\infty}\tau^{\frac{s}{2}-1}
\left(\tau-\Delta_{V}\right)^{-1}fd\tau \notag\\
&=&c(s)(-\Delta_{V})\int_{0}^{\infty}\tau^{\frac{s}{2}-1}\left[\left(
\tau-\Delta_{V}\right)^{-1}-\left(\tau-\Delta\right)^{-1}\right]fd\tau  \notag \\
&&\qquad+c(s)(-\Delta_{V}) \int_{0}^{\infty}\tau^{\frac{s}{2}-1}
\left(\tau-\Delta\right)^{-1} f d\tau.
\end{eqnarray}
Using the relations
\begin{eqnarray*}
&&\qquad(-\Delta_{V})\left[\left(\tau-\Delta_{V}\right)^{-1}-\left(\tau-\Delta\right)^{-1}\right] \\
&&= -(-\Delta_{V})\left(\tau-\Delta_{V}\right)^{-1}V\left(\tau-\Delta\right)^{-1} \\
&&=-\left(\tau-\Delta_{V}\right) \left(\tau-\Delta_{V}\right)^{-1}V\left(\tau-\Delta\right)^{-1} + \tau \left(\tau-\Delta_{V}\right)^{-1}V\left(\tau-\Delta\right)^{-1}\\
 &&= -V\left(\tau-\Delta\right)^{-1} + \tau \left(\tau-\Delta_{V}\right)^{-1}V\left(\tau-\Delta\right)^{-1},
 \end{eqnarray*}
we find
 \begin{eqnarray*}
\left(-\Delta_{V}\right)^{\frac{s}{2}}f &= &-c(s)\int_{0}^{\infty}\tau^{\frac{s}{2}-1}   V\left(\tau-\Delta\right)^{-1} fd\tau + c(s)\int_{0}^{\infty}\tau^{\frac{s}{2}}\left(\tau-\Delta_{V}\right)^{-1}
V\left(\tau-\Delta\right)^{-1} fd\tau  \notag \\
&&\qquad-c(s)\Delta\int_{0}^{\infty}\tau^{\frac{s}{2}-1}
\left(\tau-\Delta\right)^{-1} fd\tau+c(s) \int_{0}^{\infty}\tau^{\frac{s}{2}-1}
V\left(\tau-\Delta\right)^{-1} f d\tau   \notag\\
&=&\left(-\Delta\right)^{\frac{s}{2}}f+c(s)\int_{0}^{\infty}\tau^{\frac{s}{2}}\left(
\tau-\Delta_{V}\right)^{-1}V\left(\tau-\Delta\right)^{-1}fd\tau. \notag
\end{eqnarray*}
Therefore we obtain
\begin{eqnarray}
\left(-\Delta_{V}\right)^{\frac{s}{2}}f
&=&\left(-\Delta\right)^{\frac{s}{2}}f  + c(s)\int_{0}^{\infty}\tau^{\frac{s}{2}}\left(
\tau-\Delta_{V}\right)^{-1}V\left(\tau-\Delta\right)^{-1}fd\tau  \label{4.180}
\end{eqnarray}
for $0<s<2$.

By (\ref{a1}) and the  H\"{o}lder inequality with $1/q+1/r=1/2$ , we have
\begin{eqnarray}
& \quad \quad
 \int_{0}^{\infty} \left\|\tau^{\frac{s}{2}}\left(
\tau-\Delta_{V}\right)^{-1}V\left(\tau-\Delta\right)^{-1}f\right\|_{L^{2}(\mathbb{R}^{2})}d\tau   \label{estimate} \\
&\lesssim   \|f\|_{L^{q}(\mathbb{R}^{2})}\|V\|_{L^{r}(\mathbb{R}^{2})} \int_{1}^{\infty}\tau^{\frac{s}{2}}
\left\|\left(\frac{\tau}{\beta}-\Delta\right)^{-1}\right\|_{L^{2}(\mathbb{R}^{2})\rightarrow
                               L^{2}(\mathbb{R}^{2}) } \left\|\left(\tau-\Delta\right)^{-1}\right\|_{L^{q}(\mathbb{R}
                               ^{2})\rightarrow L^{q}(\mathbb{R}^{2})}d\tau \notag \\
&+  \|f\|_{L^{q}(\mathbb{R}^{2})}\|<x>^{2a}V\|_{L^{r}(\mathbb{R}^{2})} \notag \\
& \times \int_{0}^{1}\tau^{\frac{s}{2}}
\left\|\left(\frac{\tau}{\beta}-\Delta\right)^{-1}<x>^{-a}\right\|_{L^{2}(\mathbb{R}^{2})\rightarrow
                               L^{2}(\mathbb{R}^{2}) } \left\|<x>^{-a}\left(\tau-\Delta\right)^{-1}\right\|_{L^{q}(\mathbb{R}
                               ^{2})\rightarrow L^{q}(\mathbb{R}^{2})}d\tau, \notag
                               \end{eqnarray}
here $ a$ is chosen so that $ a > 1-s/2.$ Now taking
$$ a= 1-\frac{s}{2} + \varepsilon, \ \   s_0 = \frac{1}2 + \frac{s}4 - \frac{\varepsilon}4, $$
where $0<\varepsilon<\frac{s}{2} $.
we can apply Lemma \ref{Lemma Resolvent Estimate 2 }, since
$$ \frac{a}2 + s_0 = 1 + \varepsilon \left( \frac{1}2 - \frac{1}4 \right) > 1 $$
and get
\begin{eqnarray}
& &\int_{0}^{\infty} \left\|\tau^{\frac{s}{2}}\left(
\tau-\Delta_{V}\right)^{-1}V\left(\tau-\Delta\right)^{-1}f\right\|_{L^{2}(\mathbb{R}^{2})}d\tau \\
&\lesssim & \|f\|_{L^{q}(\mathbb{R}^{2})}\|V\|_{L^{r}(\mathbb{R}^{2})} \int_{1}^{\infty}\tau^{\frac{s}{2}-2}d\tau + \|f\|_{L^{q}(\mathbb{R}^{2})}\|<x>^{2a}V\|_{L^{r}(\mathbb{R}^{2})} \int_{0}^{1}\tau^{\frac{s}{2}-2 s_0}d\tau \notag \\
&\lesssim & \|f\|_{L^{q}(\mathbb{R}^{2})}, \notag
\end{eqnarray}
since
$$ \frac{s}2 - 2 s_0 = -1 + \frac{\varepsilon}2 > -1$$
and
$$ r(3-2a)=r(1+s-2\varepsilon)>2.$$
\end{proof}
\begin{lemma}
\label{Lemma Jv}
Let $V(x)$ satisfy $(\rm{H1})$ and $(\rm{H2})$. Then we have the estimates
\begin{description}
  \item[a)] for any $0 \leq  s < 1$ we have
  \begin{equation}\label{eq.n.m11}
    \left\|\left(-\Delta_{V}\right)^{\frac{s}{2}}f\right\|_{L^{2}(\mathbb{R}^{2})}  \lesssim  \left\|\left(-\Delta \right)^{\frac{s}{2}}f\right\|_{L^{2}(\mathbb{R}^{2})};
  \end{equation}
  \item[b)] for any $ 1 \leq s <2,$ and $0 < \sigma < 1,$  we have
  \begin{eqnarray} \label{n.m0}
\left\|\left(-\Delta_{V}\right)^{\frac{s}{2}}f-\left(-\Delta \right)^{\frac{s}{2}}f\right\|_{L^{2}(\mathbb{R}^{2})}\lesssim
\|(-\Delta )^{\frac{s}{2}}f\|_{L^{2}(\mathbb{R}^{2})}^{\frac{\sigma}{s}}\|f\|_{L^{2}(\mathbb{R}^{2})}^{1-\frac{\sigma}{s}}.
\end{eqnarray}
\end{description}
\end{lemma}

\begin{proof}
We have (\ref{eq.n.m11}) by Lemma \ref{Lemma Jv0}
and the Sobolev embedding
\begin{eqnarray} \label{add1}
 \|f\|_{L^{q}(\mathbb{R}^{2})}  \lesssim \|(-\Delta)^{\frac{s}{2}}f\|_{L^2(\mathbb{R}^2)},
 \end{eqnarray}
where $\frac{s}{2}= \frac{1}{2}-\frac{1}{q}$ and $q>2$.\\

Applying  the Sobolev embedding (\ref{add1}) with $\sigma=s,$
where $\frac{\sigma}{2}=\frac{1}{2}-\frac{1}{q}$ and $q>2,$
and the interpolation inequality
$$ \|(-\Delta)^{\frac{\sigma}{2}}f\|_{L^2(\mathbb{R}^2)} \lesssim \|(-\Delta)^{\frac{s}{2}}f\|_{L^2(\mathbb{R}^2)}^\theta \|f\|^{1-\theta}_{L^2(\mathbb{R}^2)},$$
with
$\theta=\frac{\sigma }{s}$, we get \eqref{n.m0} from Lemma \ref{Lemma Jv0}.

\end{proof}
\begin{remark}
For any $1<s<2$, we have
 \begin{eqnarray} \label{n.m1}
\left\|\left(-\Delta_{V}\right)^{\frac{s}{2}}f-\left(-\Delta \right)^{\frac{s}{2}}f\right\|_{L^{2}(\mathbb{R}^{2})}\lesssim
\|(-\Delta_V)^{\frac{s}{2}}f\|_{L^{2}(\mathbb{R}^{2})}^{\frac{\sigma}{s}}\|f\|_{L^{2}(\mathbb{R}^{2})}^{1-\frac{\sigma}{s}},
\end{eqnarray}
due to \eqref{n.m0} and Lemma \ref{lemma 4.2}.
\end{remark}
By Lemma \ref{lemma 4.2} and (\ref{eq.n.m11}) in Lemma \ref{Lemma Jv},
we have the following equivalence  property directly.
 \begin{lemma}
 \label{equi lemma}
Suppose that $\rm{(H1)}$ and $\rm{(H2)}$ are satisfied.
For any $0\leq s<1$, we have
\begin{eqnarray} \label{equi}
\left\|\left(-\Delta_{V}\right)^{\frac{s}{2}}f\right\|_{L^{2}(\mathbb{R}^{2})}\sim
\|(-\Delta )^{\frac{s}{2}}f\|_{L^{2}(\mathbb{R}^{2})}.
\end{eqnarray}
\end{lemma}

To estimate $\left\|(-\Delta_{V})^{\frac{s}{2}}M(t)\left(|u|^{p-1}u\right)\right\|_{ L^{2}(\mathbb{R}^{2})}$ with $1<s<2$,
we need the estimate about $\left\|u\right\|_{L^{\infty}(\mathbb{R}^{2})}$.
We consider the following lemma.
\begin{lemma}
\label{Lemma 5.10}
Suppose that $\rm{(H1)}$ and $\rm{(H2)}$ are satisfied.
 Then
for any $1<s<2$, we have
\begin{eqnarray}  \label{n.m10}
\left\|f\right\|_{L^{\infty}(\mathbb{R}^{2})} \lesssim \|(-\Delta_V)^{\frac{s}{2}}f\|^{\frac{1}{s}}_{L^{2}(\mathbb{R}^{2})}
                                                     \left\|f\right\|^{1-\frac{1}{s}}_{L^{2}(\mathbb{R}^{2})}
\end{eqnarray}
\end{lemma}
\begin{proof}
By the H\"{o}lder  inequality, we have
\begin{eqnarray}
 \|\mathcal{F}f\|_{L^{1}(\mathbb{R}^{2})}& \lesssim&\tau \|\mathcal{F}f\|_{L^{2}(|\xi|\leq \tau)}  \notag \\
&&\quad + \||\xi|^{s}\mathcal{F}f\|_{L^{2}(|\xi|\geq \tau)}\||\xi|^{-s}\|_{L^{2}(|\xi|\geq \tau)} \label{FV1} \\
& \lesssim&  \tau \|f\|_{L^{2}(\mathbb{R}^{2})}+\sqrt{\frac{1}{2(s-1)}} \tau^{1-s}\left\|(-\Delta)^{\frac{s}{2}}f\right\|_{L^{2}(\mathbb{R}^{2})} \notag
\end{eqnarray}
for any $s>1$ and $\tau >0$.
Let $\tau=\left( \sqrt{\frac{1}{2(s-1)}} \|(-\Delta)^{\frac{s}{2}}f\|_{L^{2}(\mathbb{R}^{2})}
 \right)^{\frac{1}{s}}\|f\|_{L^{2}(\mathbb{R}^{2})}^{-\frac{1}{s}}$.
Then we have
\begin{eqnarray}
 \tau \|f\|_{L^{2}(\mathbb{R}^{2})}&=&\sqrt{\frac{1}{2(s-1)}}\tau^{1-s}\left\|(-\Delta)^{\frac{s}{2}}f\right\|_{L^{2}(\mathbb{R}^{2})} \label{FV2}\\
&=& \left(\frac{1}{2(s-1)}\right)^{\frac{1}{2s}}\|(-\Delta)^{\frac{s}{2}}f\|^{\frac{1}{s}}_{L^{2}(\mathbb{R}^{2})}
                                                     \left\|f\right\|^{1-\frac{1}{s}}_{L^{2}(\mathbb{R}^{2})} . \notag
\end{eqnarray}
By (\ref{FV1}), (\ref{FV2}) and Lemma \ref{lemma 4.2}, we have our desired result.
\end{proof}
\begin{remark}
\label{Remark 6.1}
By Lemma \ref{Lemma 5.10} and $|J_{V}|^{s}(t)=M(-t)(-t^{2}\Delta_{V})^{\frac{s}{2}}M(t)$,  we have
\begin{eqnarray}
\left\|f(t,\cdot)\right\|_{L^{\infty}(\mathbb{R}^{2})} & \lesssim& \|(-\Delta_{V})^{\frac{s}{2}}M(t)f(t, \cdot)\|^{\frac{1}{s}}_{L^{2}(\mathbb{R}^{2})} \left\|M(t)f(t,\cdot)\right\|^{1-\frac{1}{s}}_{L^{2}(\mathbb{R}^{2})} \label{16.1} \\
& \lesssim&  t^{-1} \||J_{V}|^{s}(t)f(t, \cdot)\|^{\frac{1}{s}}_{L^{2}(\mathbb{R}^{2})} \left\|f(t,\cdot)\right\|^{1-\frac{1}{s}}_{L^{2}(\mathbb{R}^{2})} \notag
\end{eqnarray}
for $1<s<2$.
\end{remark}
\section {Proof of Theorem \ref{main theorem}}
\label{Proof}
We define the function space $X_{T}$ as follows
\begin{eqnarray}
X_{T}=\left\{f\in C\left([1,T];\mathcal{S}^{\prime}\right); |||f|||_{X_{T}}
=\left\||J_{V}|^{\alpha}f\right\|_{L^{\infty}\left([1,T];L^{2}(\mathbb{R}^{2})\right)}
+\sup_{t\in[1,T]}\|f\|_{L^{2}(\mathbb{R}^{2})}\notag
\right\},
\end{eqnarray}
where $T>1$ and $1<\alpha<2$.
Since we can obtain the local existence of solutions to the equation (\ref{NLS}) by the standard contraction mapping principle, we skip the proof in this section.
Multiplying both sides of  the equation (\ref{NLS}) by $|J_{V}|^{\alpha}$ and using Proposition \ref{proposition 1.1}, we have
 \begin{eqnarray}
& &  \hspace {1cm}\left(i \partial_t  + \frac{1}{2}\Delta_V\right)|J_{V}|^{\alpha}u \label{1.15} \\
&&   \hspace {0.5cm}   = it^{\alpha-1}M(-t)A(\alpha)M(t)u+\lambda|J_{V}|^{\alpha}\left( |u|^{p-1}u\right).\notag
 \end{eqnarray}
Let $ |J_{V}|^{\alpha}u=u_{\alpha}^{\uppercase\expandafter{\romannumeral 1}}+u_{\alpha}^{\uppercase\expandafter{\romannumeral 2}}$.
We consider
\begin{eqnarray}
\left\{
\begin{array}{l}
\left(i \partial_t  + \frac{1}{2}\Delta_{V}\right) u_{\alpha}^{\uppercase\expandafter{\romannumeral 1}} = \lambda |J_{V}|^{\alpha}\left( |u|^{p-1}u\right), \\
u_{\alpha}^{\uppercase\expandafter{\romannumeral 1}}(1) = |J_{V}|^{\alpha}(1)u_{0},
\end{array}
\right. \label{NLS1}
\end{eqnarray}
and
\begin{eqnarray}
\left\{
\begin{array}{l}
\left(i \partial_t  + \frac{1}{2}\Delta_{V}\right) u_{\alpha}^{\uppercase\expandafter{\romannumeral 2}} = it^{\alpha-1}M(-t)A(\alpha)M(t)u, \\
u_{\alpha}^{\uppercase\expandafter{\romannumeral 2}}(1) = 0
\end{array}
\right. \label{NLS2}
\end{eqnarray}
for  $p>2$, and $t\geq 1$, where $u=u(t,x)$ is a real valued unknown function, $x \in \mathbb{R}^{2}$,
$\Delta_{V}=\Delta-V(x)$, $|J_{V}|^{\alpha}(t)=M(-t)\left(-t^{2}\Delta_{V}\right)^{\frac{\alpha}{2}}M(t)$, $M(t)=e^{-\frac{i}{2t}|x|^{2}}$
 and $A(\alpha)=\alpha\left(-\Delta_{V}\right)^{\frac{\alpha}{2}}+\left[x\cdot\nabla,\left(-\Delta_{V}\right)^{\frac{\alpha}{2}}\right]$.

First we consider the integral equation
\begin{eqnarray}
u_{\alpha}^{\uppercase\expandafter{\romannumeral 1}}=e^{\frac{i}{2}t\Delta_{V}}e^{-\frac{i}{2}\Delta_{V}}|J_{V}|^{\alpha}(1)u_{0}
-i\lambda\int_{1}^{t}e^{\frac{i}{2}(t-\tau)\Delta_{V}}|J_{V}|^{\alpha}\left(|u|^{p-1}u\right)(\tau)d\tau \label{NLS3}
\end{eqnarray}
associated with (\ref{NLS1}).
For simplicity, we let $|J_{V}|^{\alpha}\left(|u|^{p-1}u\right)=F_{\alpha}.$ Then from (\ref{NLS3}) we have
\begin{eqnarray}
u_{\alpha}^{\uppercase\expandafter{\romannumeral 1}}=e^{\frac{i}{2}t\Delta_{V}}e^{-\frac{i}{2}\Delta_{V}}|J_{V}|^{\alpha}(1)u_0-i\lambda\int_{1}^{t}e^{\frac{i}{2}(t-\tau)\Delta_{V}}F_{\alpha}(\tau)d\tau. \label{NLS4}
\end{eqnarray}
We also have
\begin{eqnarray}
u_{\alpha}^{\uppercase\expandafter{\romannumeral 2}}=\int_{1}^{t}e^{\frac{i}{2}(t-\tau)\Delta_{V}}\tau^{\alpha-1}M(-\tau)A(\alpha)M(\tau)u(\tau)d\tau  \label{NLS8}
\end{eqnarray}
from (\ref{NLS2}).\\
By Proposition \ref{proposition 2.2} and Proposition \ref {theorem 2.2},
from (\ref{NLS4}) we have
\begin{eqnarray}
\|u_{\alpha}^{\uppercase\expandafter{\romannumeral 1}}\|_{L^{\infty}\left([1,T]; L^{2}(\mathbb{R}^{2})\right)}
 \lesssim\||J_{V}|^{\alpha}(1)u_0\|_{L^{2}(\mathbb{R}^{2})}+\left\|F_{\alpha}\right\|_{L^{1}\left([1,T]; L^{2}(\mathbb{R}^{2})\right)},
\label{4.7}
\end{eqnarray}
where  $F_{\alpha}=|J_{V}|^{\alpha}\left(|u|^{p-1}u\right)$, $|J_{V}|^{\alpha}(t)=M(-t)\left(-t^{2}\Delta_{V}\right)^{\frac{\alpha}{2}}M(t)$, and $M(t)=e^{-\frac{i}{2t}|x|^{2}}$.\\
By (\ref{n.m0}) in Lemma \ref{Lemma Jv}, and Lemma 3.4 in \cite{GOV}, we obtain
\begin{eqnarray}
&& \quad \left\|F_{\alpha}\right\|_{ L^{2}(\mathbb{R}^{2})}  \label{7.10} \\
&& =\left\||J_{V}|^{\alpha}\left(|u|^{p-1}u\right)\right\|_{ L^{2}(\mathbb{R}^{2})}  \notag \\
&& \lesssim \left(\left\| |J|^{\alpha}(|u|^{p-1}u)\right\|_{ L^{2}(\mathbb{R}^{2})}
              +t^{\alpha-\sigma}\left\|  |J|^{\alpha}(|u|^{p-1}u)\right\|_{ L^{2}(\mathbb{R}^{2})}^{\frac{\sigma }{\alpha}}
              \left\| |u|^{p-1}u \right\|_{ L^{2}(\mathbb{R}^{2})}^{1-\frac{\sigma }{\alpha}}
               \right) \notag \\
&& \lesssim \|u\|^{p-1}_{ L^{\infty}(\mathbb{R}^{2})} \left(\left\| |J|^{\alpha}u\right\|_{ L^{2}(\mathbb{R}^{2})} +t^{\alpha-\sigma} \left\||J|^{\alpha}u\right\|_{ L^{2}(\mathbb{R}^{2})}^{\frac{\sigma }{\alpha}} \left\| u\right\|_{ L^{2}(\mathbb{R}^{2})}^{1-\frac{\sigma }{\alpha}}\right)  \notag
\end{eqnarray}
for $p> 2$, $1<\alpha<2$ and $0<\sigma <1.$
By Lemma \ref{lemma 4.2}, from  (\ref{7.10}) we obtain
\begin{eqnarray}
&& \quad \left\|F_{\alpha}\right\|_{ L^{2}(\mathbb{R}^{2})}  \label{7.2} \\
&& \lesssim \|u\|^{p-1}_{ L^{\infty}(\mathbb{R}^{2})} \left(\left\| |J_{V}|^{\alpha}u\right\|_{ L^{2}(\mathbb{R}^{2})} +t^{\alpha-\sigma}\left\||J_V|^{\alpha}u\right\|_{ L^{2}(\mathbb{R}^{2})}^{\frac{\sigma }{\alpha}}
\left\|u\right\|_{ L^{2}(\mathbb{R}^{2})}^{1-\frac{\sigma }{\alpha}}
\right)  \notag
\end{eqnarray}
for $p> 2$, $1<\alpha<2$ and $0<\sigma < 1.$
By (\ref{16.1}) in Remark \ref{Remark 6.1},
from  (\ref{7.2}) we get
\begin{eqnarray}
&& \quad \left\|F_{\alpha}\right\|_{ L^{2}(\mathbb{R}^{2})}  \label{7.3} \\
&&  \lesssim\left(  t^{-1}\||J_{V}|^{\alpha}u\|^{\frac{1}{\alpha}}_{L^{2}(\mathbb{R}^{2})} \left\|u\right\|^{1-\frac{1}{\alpha}}_{L^{2}(\mathbb{R}^{2})}  \right)^{p-1} \notag \\
&& \quad \times \left(\left\| |J_{V}|^{\alpha}u\right\|_{ L^{2}(\mathbb{R}^{2})} +t^{\alpha-\sigma} \left\||J_V|^{\alpha}u\right\|_{ L^{2}(\mathbb{R}^{2})}^{
\frac{\sigma }{\alpha}} \left\|u\right\|_{ L^{2}(\mathbb{R}^{2})}^{1-\frac{\sigma }{\alpha}}
\right). \notag
\end{eqnarray}
Then we obtain
\begin{eqnarray}
&& \label{i1}  \left\|F_{\alpha}\right\|_{L^{1}\left([1,T]; L^{2}(\mathbb{R}^{2})\right)} \\
&&\lesssim
\|u_0\|_{L^{2}(\mathbb{R}^{2})}^{\left(1-\frac{1}{\alpha}\right)(p-1)}
\||J_{V}|^{\alpha}u\|_{L^{\infty}\left([1,T]; L^{2}(\mathbb{R}^{2})\right)}^{\frac{p-1}{\alpha}+1}
\notag \\
&& \quad +  \|u_0\|_{L^{2}(\mathbb{R}^{2})}^{(p-1)\left(1-\frac{1}{\alpha}\right)+1-\frac{\sigma }{\alpha}}\left\|t^{-(p-1)+\alpha-\sigma}\||J_{V}|^{\alpha}u\|_{L^{2}(\mathbb{R}^{2})}^{\frac{p-1+\sigma}{\alpha}}
                    \right\|_{L^{1}\left([1,T]\right)} \notag\\
&& \lesssim
\|u_0\|_{H^{\alpha}(\mathbb{R}^{2})}^{\left(1-\frac{1}{\alpha}\right)(p-1)}
\||J_{V}|^{\alpha}u\|_{L^{\infty}\left([1,T]; L^{2}(\mathbb{R}^{2})\right)}^{\frac{p-1}{\alpha}+1}\notag  \\
&& \quad + \|u_0\|_{H^{\alpha}(\mathbb{R}^{2})}^{(p-1)\left(1-\frac{1}{\alpha}\right)+1-\frac{\sigma }{\alpha}}
\|t^{-p+1+\alpha-\sigma}\|_{L^{1}\left([1,T]\right)}
\left\||J_{V}|^{\alpha}u \right\|_{L^{\infty}\left([1,T];L^{2}(\mathbb{R}^{2})\right)}^{\frac{p-1+\sigma}{\alpha}} \notag \\
&& \lesssim
\|u_0\|_{H^{\alpha}(\mathbb{R}^{2})}^{\left(1-\frac{1}{\alpha}\right)(p-1)}
\||J_{V}|^{\alpha}u\|_{L^{\infty}\left([1,T]; L^{2}(\mathbb{R}^{2})\right)}^{\frac{p-1}{\alpha}+1}\notag  \\
&& \quad + \|u_0\|_{H^{\alpha}(\mathbb{R}^{2})}^{(p-1)\left(1-\frac{1}{\alpha}\right)+1-\frac{\sigma }{\alpha}}
\left\||J_{V}|^{\alpha}u \right\|_{L^{\infty}\left([1,T];L^{2}(\mathbb{R}^{2})\right)}^{\frac{p-1+\sigma}{\alpha}}, \notag
\end{eqnarray}
since  we can choose $\alpha$ and $\sigma$
such that $-p+2+\alpha-\sigma<0$ for $p>2$, where $1<\alpha < \frac{3}{2}$ and $\frac{2}{3}<\sigma<1.$
By (\ref{n.m0}) in Lemma \ref{Lemma Jv} and (\ref{i1}), then we have
\begin{eqnarray}
&& \quad \|u_{\alpha}^{\uppercase\expandafter{\romannumeral 1}}\|_{L^{\infty}\left([1,T]; L^{2}(\mathbb{R}^{2})\right)} \label{7.11} \\
&& \lesssim \left\||J|^{\alpha}(1)u_0\right\|_{L^{2}(\mathbb{R}^{2})}
+\left\||J|^{\alpha}(1)u_0\right\|_{L^{2}(\mathbb{R}^{2})}^{\frac{\sigma}{\alpha}}\|u_0\|_{L^{2}(\mathbb{R}^{2})}^{1-\frac{\sigma}{\alpha}} \notag \\
&&\quad +\|u_0\|_{H^{\alpha}(\mathbb{R}^{2})}^{\left(1-\frac{1}{\alpha}\right)(p-1)}
\||J_{V}|^{\alpha}u\|_{L^{\infty}\left([1,T]; L^{2}(\mathbb{R}^{2})\right)}^{\frac{p-1}{\alpha}+1}\notag  \\
&& \quad + \|u_0\|_{H^{\alpha}(\mathbb{R}^{2})}^{(p-1)\left(1-\frac{1}{\alpha}\right)+1-\frac{\sigma }{\alpha}}
\left\||J_{V}|^{\alpha}u \right\|_{L^{\infty}\left([1,T];L^{2}(\mathbb{R}^{2})\right)}^{\frac{p-1+\sigma}{\alpha}} \notag
\end{eqnarray}
for $p>2$, where $1<\alpha < \frac{3}{2}$ and $\frac{2}{3}<\sigma<1$.

By Proposition \ref{theorem 2.2}, we have from (\ref{NLS8})
\begin{eqnarray}
\qquad \left\|u_{\alpha}^{\uppercase\expandafter{\romannumeral 2}}\right\|_{L^{\infty}\left([1,T];L^{2}(\mathbb{R}^{2})\right)}\lesssim
\left\|t^{\alpha-1}M(-t)A(\alpha)M(t)u\right\|_{L^{p_{1}^{\prime}}\left([1,T];L^{q_{1}^{\prime}}(\mathbb{R}^{2})\right)} \label{17.12}
\end{eqnarray}
for $1<\alpha<2$, where $M(t)=e^{-\frac{i}{2t}|x|^{2}},$
 $A(\alpha)=\alpha\left(-\Delta_{V}\right)^{\frac{\alpha}{2}}+\left[x\cdot\nabla,\left(-\Delta_{V}\right)^{\frac{\alpha}{2}}\right]$, $\frac{1}{p_{1}}+\frac{1}{q_{1}}=\frac{1}{2}$,
 $\frac{1}{p_{1}}+\frac{1}{p_{1}^{\prime}}=1,$ $\frac{1}{q_{1}}+\frac{1}{q_{1}^{\prime}}=1,$
 and $1< q_1^{\prime} < 2.$
Let $1<\alpha<\frac{3}{2}$. By Lemma \ref{Lemma A(s)} and the Sobolev inequality
\begin{eqnarray}
\|u\|_{L^{q_1}(\mathbb{R}^{2})}\lesssim \|(-\Delta)^{\frac{\alpha}{2}}u\|_{L^{2}(\mathbb{R}^{2})}^{\theta}
\|u\|_{L^{2}(\mathbb{R}^{2})}^{1-\theta}    \notag
\end{eqnarray}
for $0 <  \theta < \frac{1}{\alpha},$ where $q_1=\frac{2}{1-\theta \alpha}$,
from (\ref{17.12}) we have
\begin{eqnarray}
 && \quad \left\|u_{\alpha}^{\uppercase\expandafter{\romannumeral 2}}\right\|_{L^{\infty}\left([1,T];L^{2}(\mathbb{R}^{2})\right
)} \label{17.13}  \\
&\lesssim &\left\|t^{\alpha-1}M(t)u\right\|_{L^{p_{1}^{\prime}}\left([1,T];L^{q_{1}}(\mathbb{R}^{2})\right)} \notag  \\
&\lesssim &\left\|t^{\alpha-1} \|(-\Delta)^{\frac{\alpha}{2}}M(t)u\|_{L^{2}(\mathbb{R}^{2})}^{\theta}
\|M(t)u\|_{L^{2}(\mathbb{R}^{2})}^{1-\theta}
\right\|_{L^{p_{1}^{\prime}}([1,T])} \notag \\
&\lesssim  &\|u_0\|_{L^{2}(\mathbb{R}^{2})}^{1-\theta}
\left\|t^{\alpha-1} \|(-\Delta)^{\frac{\alpha}{2}}M(t)u\|_{L^{2}(\mathbb{R}^{2})}^{\theta}
\right\|_{L^{p_{1}^{\prime}}([1,T])} \notag \\
&\lesssim &\|u_0\|_{H^{\alpha}(\mathbb{R}^{2})}^{1-\theta}
\left\|t^{\alpha-1-\alpha\theta} \| |J|^{\alpha}u\|_{L^{2}(\mathbb{R}^{2})}^{\theta}
\right\|_{L^{p_{1}^{\prime}}([1,T])}   \notag \\
&\lesssim &\|u_0\|_{H^{\alpha}(\mathbb{R}^{2})}^{1-\theta}
\|t^{\alpha-1-\alpha\theta} \|_{L^{p_{1}^{\prime}}([1,T])}
\left\| |J|^{\alpha}u \right\|_{L^{\infty}([1,T]; L^{2}(\mathbb{R}^{2}))}^{\theta} \notag
\end{eqnarray}
for $0< \theta < \frac{1}{\alpha}$, where $\frac{1}{p_{1}}+\frac{1}{p_{1}^{\prime}}=1, \frac{1}{q_{1}}+\frac{1}{q_{1}^{\prime}}=1,
\frac{1}{p_{1}}+\frac{1}{q_{1}}=\frac{1}{2},$ $q_1=\frac{2}{1-\theta \alpha}$ and $1< q_{1}^{\prime} < 2$.
For $1<\alpha<\frac{3}{2},$ we choose $\theta \in \left(\frac{2}{3}, \frac{1}{\alpha}\right)$.
Then we have $p_1^{\prime}=\frac{2}{2-\alpha \theta}.$
Since $\left(\alpha-1-\alpha \theta\right)p_{1}^{\prime}+1=-\frac{\alpha (3\theta-2)}{2-\alpha \theta}<0$ for $1<\alpha<\frac{3}{2}$, where  $\frac{2}{3}<\theta <\frac{1}{\alpha},$ we have
\begin{eqnarray}
\|t^{\alpha-1-\alpha \theta} \|_{L^{p_{1}^{\prime}}([1,T])} \leq C. \label{i2}
\end{eqnarray}
By (\ref{17.13}), (\ref{i2}) and Lemma \ref{lemma 4.2}, we have
\begin{eqnarray}
 && \quad \left\|u_{\alpha}^{\uppercase\expandafter{\romannumeral 2}}\right\|_{L^{\infty}\left([1,T];L^{2}(\mathbb{R}^{2})\right
)} \label{7.14}  \\
&\lesssim&\|u_0\|_{H^{\alpha}(\mathbb{R}^{2})}^{\frac{1}{4}}
\left\| |J_{V}|^{\alpha}u \right\|_{L^{\infty}([1,T];L^{2}(\mathbb{R}^{2}))}^{\frac{3}{4}} \notag
\end{eqnarray}
for $1<\alpha < \frac{3}{2}.$

Using (\ref{7.11}) and (\ref{7.14}),  we have
\begin{eqnarray}
\left\||J_{V}|^{\alpha}u\right\|_{L^{\infty}\left([1,T];L^{2}(\mathbb{R}^{2})\right
)} &\leq&
\left\|u_{\alpha}^{\uppercase\expandafter{\romannumeral 1}}\right\|_{L^{\infty}\left([1,T];L^{2}(\mathbb{R}^{2})\right
)}
+
\left\|u_{\alpha}^{\uppercase\expandafter{\romannumeral 2}}\right\|_{L^{\infty}\left([1,T];L^{2}(\mathbb{R}^{2})\right)} \notag\\
& \lesssim &\left\|u_{0}\right\|_{H^{\alpha}(\mathbb{R}^2)\cap \dot{H}^{0,\alpha}(\mathbb{R}^2)} \notag \\
&&\quad +\|u_{0}\|_{H^{\alpha}(\mathbb{R}^2)\cap \dot{H}^{0,\alpha}(\mathbb{R}^2)}^{\left(1-\frac{1}{\alpha}\right)(p-1)}
\||J_{V}|^{\alpha}u\|_{L^{\infty}\left([1,T]; L^{2}(\mathbb{R}^{2})\right)}^{\frac{p-1}{\alpha}+1}\notag  \\
&& \quad + \|u_{0}\|_{H^{\alpha}(\mathbb{R}^2)\cap \dot{H}^{0,\alpha}(\mathbb{R}^2)}^{(p-1)\left(1-\frac{1}{\alpha}\right)+1-\frac{\sigma }{\alpha}}
\left\||J_{V}|^{\alpha}u \right\|_{L^{\infty}\left([1,T];L^{2}(\mathbb{R}^{2})\right)}^{\frac{p-1+\sigma}{\alpha}}\notag\\
&&\quad +\|u_{0}\|_{H^{\alpha}(\mathbb{R}^2)\cap \dot{H}^{0,\alpha}(\mathbb{R}^2)}^{\frac{1}{4}}
\left\| |J_{V}|^{\alpha}u \right\|_{L^{\infty}([1,T];L^{2}(\mathbb{R}^{2}))}^{\frac{3}{4}}\notag
\end{eqnarray}
for $p>2$,
where $1<\alpha < \frac{3}{2}$ and $\frac{2}{3}<\sigma<1$.
Then for a fixed $C>0$ we have
\begin{eqnarray}
\left\||J_{V}|^{\alpha}u\right\|_{L^{\infty}\left([1,T];L^{2}(\mathbb{R}^{2})\right
)}
\leq C\|u_{0}\|_{H^{\alpha}(\mathbb{R}^2)\cap \dot{H}^{0,\alpha}(\mathbb{R}^2)},
\end{eqnarray}
if $\|u_{0}\|_{H^{\alpha}(\mathbb{R}^2)\cap \dot{H}^{0,\alpha}(\mathbb{R}^2)}$ is small enough.
By a standard continuity argument and Remark \ref{Remark 6.1},
we have the time decay estimate (\ref{timedecay}) if
$\epsilon_{0}$ is small enough.
From (\ref{NLS}), we have
\begin{eqnarray*}
u(t)=e^{\frac{i}{2}t\Delta_V}\left(e^{-\frac{i}{2}\Delta_V}u_0-i\lambda\int_1^t e^{-\frac{i}{2}\tau\Delta_V}(|u|^{p-1}u)(\tau)d\tau\right).
\end{eqnarray*}
Let $u_+=e^{-\frac{i}{2}\Delta_{V}}u_0-i\lambda \int_{1}^{\infty}e^{-\frac{i}{2}\tau\Delta_{V}}(|u|^{p-1}u)(\tau)d\tau$. The we have
\begin{eqnarray*}
u(t)=e^{\frac{i}{2}t\Delta_V}u_+ +i\lambda\int_t^\infty e^{-\frac{i}{2}(t-\tau)\Delta_V}(|u|^{p-1}u)(\tau)d\tau.
\end{eqnarray*}
We obtain the scattering (\ref{scattering}) by a standard argument
from the time decay estimate (\ref{timedecay}). We omit the proof here.
\section{Appendix I}
\label{Appendix}
Let $M(t)=e^{-\frac{i}{2t}|x|^{2}}$
and $[B, D]=BD-DB$. To prove Proposition \ref{proposition 1.1} and Proposition \ref{proposition 1.01},
we consider the following lemmas (see \cite{CGV2014}).
\begin{lemma}\label{lemma 1.1}
We have the following identities:
\begin{equation}
[i\partial_{t},M(-t)]=\frac{|x|^{2}}{2t^{2}}M(-t),\label{1.2}
\end{equation}
and
\begin{equation}
[i\partial_{t},M(t)]=-\frac{|x|^{2}}{2t^{2}}M(t).\label{1.3}
\end{equation}
\end{lemma}
\begin{proof}
Since
\begin{eqnarray}
&&\qquad [i\partial_{t},M(-t)]f  \notag \\
& &=i\partial_{t}(M(-t)f)-M(-t)i\partial_{t}f\notag\\
& &= \frac{|x|^{2}}{2t^{2}}M(-t)f, \notag
\end{eqnarray}
then we obtain the first identity (\ref{1.2}).\\
By using the similar method, we get the second identity (\ref{1.3}).
\end{proof}

\begin{lemma}\label{lemma 1.2}
We have
\begin{equation}
[ \Delta,M(-t)]=M(-t)\left(\frac{in}{t}-\frac{|x|^{2}}{t^{2}}+2\frac{ix\cdot\nabla}{t}\right),\label{1.4}
\end{equation}
and
\begin{equation}
[ \Delta,M(t)]=M(t)\left(-\frac{in}{t}-\frac{|x|^{2}}{t^{2}}-2\frac{ix\cdot\nabla}{t}\right),\label{1.5}
\end{equation}
where  $n$ is the  generic space dimension.
\end{lemma}
\begin{proof}
By some calculations, we have
\begin{eqnarray}
&&\qquad [\Delta,M(-t)]f              \notag \\
& &=\Delta(M(-t)f)-M(-t)\Delta f      \notag\\
& &= M(-t) \Delta f +\Delta(M(-t))f+2\nabla M(-t)\cdot \nabla f -M(-t)\Delta f  \notag \\
& &= \Delta(M(-t))f+2\nabla M(-t)\cdot \nabla f  \notag \\
& &=M(-t)\left(\frac{in}{t}-\frac{|x|^{2}}{t^{2}}+2\frac{ix\cdot\nabla}{t}\right)f. \notag
\end{eqnarray}
Taking complex conjugates, we get the second identity (\ref{1.5}).
\end{proof}

We have the following commutator relations
\begin{lemma}\label{lemma 1.3}
\begin{equation}
\left[ i\partial_{t}+\frac{1}{2}\Delta,M(-t)\right]=\frac{1}{2}M(-t)\left(\frac{in}{t}+2\frac{ix \cdot \nabla}{t}\right),\label{1.6}
\end{equation}
and
\begin{equation}
\left[  i\partial_{t}+\frac{1}{2}\Delta,M(t)\right]=M(t)\left(-\frac{in}{2t}-\frac{|x|^{2}}{t^{2}}-\frac{ix\cdot\nabla}{t}\right),\label{1.7}
\end{equation}
where   $n$ is the  generic space  dimension.
\end{lemma}
\begin{proof}
By Lemmas \ref{lemma 1.1} and  \ref{lemma 1.2} , we have
\begin{eqnarray}
&&\qquad \left [ i\partial_{t}+\frac{1}{2}\Delta,M(-t)\right]f              \notag \\
& &= [ i\partial_{t},M(-t)]f + \frac{1}{2}[\Delta,M(-t)]f      \notag\\
& &= \frac{|x|^{2}}{2t^{2}}M(-t)f + \frac{1}{2}M(-t)\left(\frac{in}{t}-\frac{|x|^{2}}{t^{2}}+2\frac{ix \cdot \nabla}{t}\right)f   \notag \\
& &= \frac{1}{2}M(-t)\left(\frac{in}{t}+2\frac{ix \cdot \nabla}{t}\right) f . \notag
\end{eqnarray}
By Lemmas \ref{lemma 1.1} and  \ref{lemma 1.2} , we  also get the commutator relation (\ref{1.7}).
\end{proof}

\begin{lemma}\label{lemma 1.4}
Let $\Delta_{V}=\Delta-V(x)$. For $s \geq 0$, we have
\begin{equation}
\left[ i\partial_{t}+\frac{1}{2}\Delta_{V},\left(-t^{2}\Delta_{V}\right)^{\frac{s}{2}}\right]=\frac{is}{t}\left(-t^{2}\Delta_{V}\right)^{\frac{s}{2}}. \label{1.8}
\end{equation}
\end{lemma}
\begin{proof}
By the commutator relation $\left[ \left(-\Delta_{V}\right)^{\frac{s}{2}}, \Delta_{V}\right]=0,$  we have
\begin{eqnarray}
&& \qquad  \left[ i\partial_{t}+\frac{1}{2}\Delta_{V},\left(-t^{2}\Delta_{V}\right)^{\frac{s}{2}}\right]f     \label{1.9} \\
& &=  \left[ i\partial_{t},\left(-t^{2}\Delta_{V}\right)^{\frac{s}{2}}\right]f + \frac{1}{2} \left[ \Delta_{V},\left(-t^{2}\Delta_{V}\right)^{\frac{s}{2}}\right]f      \notag\\
& &=  \left[ i\partial_{t},\left(-t^{2}\Delta_{V}\right)^{\frac{s}{2}}\right]f .\notag
\end{eqnarray}
By some simple calculations, we have
\begin{eqnarray}
\label{1.10} \left[ i\partial_{t},\left(-t^{2}\Delta_{V}\right)^{\frac{s}{2}}\right]f
&=& i\partial_{t}\left[\left(-t^{2}\Delta_{V}\right)^{\frac{s}{2}}f \right]-i\left[\left(-t^{2}\Delta_{V}\right)^{\frac{s}{2}}\right]\partial_{t}f  \\
&=&\frac{ is}{t}\left(-t^{2}\Delta_{V}\right)^{\frac{s}{2}}f.   \notag
\end{eqnarray}
Combining (\ref{1.9}) and (\ref{1.10}), we have our desired result.
\end{proof}

\subsection{Proof of Proposition \ref{proposition 1.1}}
Since  $[B,DE]=[B,D]E+D[B,E]$,  then we have
\begin{eqnarray}
&& \label{1.12 }\hspace{1cm} \left[ i\partial_{t}+\frac{1}{2}\Delta_{V}, |J_{V}|^{s}(t)\right]f  \\
&&  \hspace{0.5cm} =\left[ i\partial_{t}+\frac{1}{2}\Delta_{V},M(-t)\left(-t^{2}\Delta_{V}\right)^{\frac{s}{2}}M(t)\right] f          \notag \\
& &\hspace{0.5cm} = \left[ i\partial_{t}+\frac{1}{2}\Delta,M(-t)\right]\left(-t^{2}\Delta_{V}\right)^{\frac{s}{2}}M(t) f \notag \\
 &  &  \qquad\qquad  + M(-t)\left[ i\partial_{t}+\frac{1}{2}\Delta_{V},\left(-t^{2}\Delta_{V}\right)^{\frac{s}{2}}M(t)\right] f.    \notag
\end{eqnarray}
By Lemmas \ref{lemma 1.3}, \ref{lemma 1.4} and $[B, DE]=[B, D]E+D[B, E]$, we have
\begin{eqnarray}
&&  \hspace {1cm}\left[ i\partial_{t}+\frac{1}{2}\Delta,M(-t)\right]\left(-t^{2}\Delta_{V}\right)^{\frac{s}{2}}M(t) f  \label{1.13}  \\
 &&   \qquad \qquad +  M(-t)\left[ i\partial_{t}+\frac{1}{2}\Delta_{V},\left(-t^{2}\Delta_{V}\right)^{\frac{s}{2}}M(t)\right] f         \notag \\
& & \hspace {0.5cm}=  \frac{i}{t}\left|J_{V}\right|^{s}(t)f+\frac{i}{t}M(-t)x\cdot\nabla \left(-t^{2}\Delta_{V}\right)^{\frac{s}{2}}M(t) f\notag\\
&&   \qquad \qquad +M(-t)\left[ i\partial_{t}+\frac{1}{2}\Delta_{V},\left(-t^{2}\Delta_{V}\right)^{\frac{s}{2}}\right] M(t)f  \notag \\
 & &\qquad  \qquad  +   M(-t)\left(-t^{2}\Delta_{V}\right)^{\frac{s}{2}}\left[ i\partial_{t}+\frac{1}{2}\Delta_{V},M(t)\right] f  \notag \\
 & & \hspace {0.5cm}=  \frac{i}{t}\left|J_{V}\right|^{s}(t)f+\frac{i}{t}M(-t)x\cdot\nabla \left(-t^{2}\Delta_{V}\right)^{\frac{s}{2}}M(t) f \notag\\
 && \qquad \qquad+M(-t)\frac{is}{t}\left(-t^{2}\Delta_{V}\right)^{\frac{s}{2}}M(t)f   \notag \\
 & &\qquad \qquad   +   M(-t)\left(-t^{2}\Delta_{V}\right)^{\frac{s}{2}}M(t)\left(-\frac{i}{t}-\frac{|x|^{2}}{t^{2}}-i\frac{x\cdot\nabla}{t}\right) f  \notag  \\
 & &  \hspace {0.5cm}=  \frac{is}{t}\left|J_{V}\right|^{s}(t)f+\frac{i}{t} M(-t) \left[ x\cdot\nabla, \left(-t^{2}\Delta_{V}\right)^{\frac{s}{2}}M(t)\right]f  \notag\\
 &&     \qquad \qquad     - M(-t)\left(-t^{2}\Delta_{V}\right)^{\frac{s}{2}}\frac{|x|^{2}}{t^{2}}M(t)f.  \notag
\end{eqnarray}
Using   $[B, DE]=[B, D]E+D[B, E]$ and $\left(-t^{2}\Delta_{V}\right)^{\frac{s}{2}}[x\cdot\nabla, M(t)]f=\left(-t^{2}\Delta_{V}\right)^{\frac{s}{2}}x\cdot\left(\nabla M(t)\right)f$,
 we have
 \begin{eqnarray}
 & &  \hspace {1cm}\frac{is}{t}\left|J_{V}\right|^{s}(t)f+\frac{i}{t} M(-t) \left[ x\cdot\nabla, \left(-t^{2}\Delta_{V}\right)^{\frac{s}{2}}M(t)\right]f   \label{1.14} \\
 &&     \qquad \qquad     - M(-t)\left(-t^{2}\Delta_{V}\right)^{\frac{s}{2}}\frac{|x|^{2}}{t^{2}}M(t)f  \notag\\
 &&     \hspace {0.5cm}=it^{s-1}M(-t)[s(-\Delta_{V})^{\frac{s}{2}}]M(t)f+\frac{i}{t}M(-t)
            [x\cdot \nabla, (-t^{2}\Delta_{V})^{\frac{s}{2}}]M(t)f\notag \\
 &&     \qquad \qquad +\frac{i}{t} M(-t)(-t^{2}\Delta_{V})^{\frac{s}{2}}[x\cdot \nabla, M(t)]f
            - M(-t)\left(-t^{2}\Delta_{V}\right)^{\frac{s}{2}}\frac{|x|^{2}}{t^{2}}M(t)f  \notag\\
 & &  \hspace {0.5cm}=it^{s-1}M(-t) A(s) M(t)f, \notag
\end{eqnarray}
where $A(s)=s\left(-\Delta_{V}\right)^{\frac{s}{2}}+\left[x\cdot\nabla,\left(-\Delta_{V}\right)^{\frac{s}{2}}\right]$.\\
Combining (\ref{1.12 }), (\ref{1.13}) and (\ref{1.14}), we complete the proof of (\ref{1.11}).

\subsection{Proof of Proposition \ref{proposition 1.01}}
Let $S=x\cdot \nabla$.
By the formula
\begin{eqnarray}
\left(-\Delta_{V}\right)^{\frac{s}{2}}f&=&c(s)(-\Delta_{V})\int_{0}^{\infty}\tau^{\frac{s}{2}-1}
\left(\tau-\Delta_{V}\right)^{-1}fd\tau \notag
\end{eqnarray}
for $0<s<2$, where $c(s)^{-1}=\int_{0}^{\infty} \tau^{\frac{s}{2}-1}(\tau+1)^{-1}d\tau$,
we get
\begin{eqnarray}
A(s)=s\left(-\Delta_{V}\right)^{\frac{s}{2}}+c(s)\int_{0}^{\infty}
\tau^{\frac{s}{2}-1}[S,-\Delta_{V}(\tau-\Delta_{V})^{-1}]d\tau. \label{2.16}
\end{eqnarray}
Using $[B, DE]=[B, D]E+D[B, E],$
we have
\begin{eqnarray}
&&[S,-\Delta_{V}(\tau-\Delta_{V})^{-1}] \label{2.17}\\
&=&[S,-\Delta_{V}](\tau-\Delta_{V})^{-1}-\Delta_{V}[S,(\tau-\Delta_{V})^{-1}] \notag \\
&=&[S,-\Delta_{V}](\tau-\Delta_{V})^{-1}\notag\\
&& \quad+\Delta_{V}(\tau-\Delta_{V})^{-1}[S, -\Delta_{V}](\tau-\Delta_{V})^{-1}. \notag
\end{eqnarray}
Since $\Delta_{V}=\Delta-V(x),$
we obtain
\begin{eqnarray}
&&[S,-\Delta_{V}] \label{2.18} \\
&=&[S,-\Delta]+[S,V] \notag \\
&=&-(x\cdot \nabla)\Delta+\Delta(x\cdot \nabla)+x\cdot \nabla V-V x\cdot \nabla \notag\\
&=&2\Delta+SV \notag \\
&=&2\Delta_{V}+W, \notag
\end{eqnarray}
where $W=(S+2)V$.
By (\ref{2.17}) and (\ref{2.18}), we have
\begin{eqnarray}
&&[S,-\Delta_{V}(\tau-\Delta_{V})^{-1}]  \label{2.19}\\
&=&2\Delta_{V}(\tau-\Delta_{V})^{-1}+W(\tau-\Delta_{V})^{-1}\notag \\
&&\quad+\Delta_{V}(\tau-\Delta_{V})^{-1}[S, -\Delta_{V}](\tau-\Delta_{V})^{-1} \notag\\
&=&2\Delta_{V}(\tau-\Delta_{V})^{-1}+W(\tau-\Delta_{V})^{-1}\notag \\
&&\quad+\Delta_{V}(\tau-\Delta_{V})^{-1}(2\Delta_{V}+W)(\tau-\Delta_{V})^{-1} \notag\\
&=&2\tau\Delta_{V}(\tau-\Delta_{V})^{-2}+\tau (\tau-\Delta_{V})^{-1}W(\tau-\Delta_{V})^{-1}. \notag
\end{eqnarray}
By (\ref{2.16}) and (\ref{2.19}), we get
\begin{eqnarray}
A(s)&=&s\left(-\Delta_{V}\right)^{\frac{s}{2}}+2c(s)\int_{0}^{\infty}
\tau^{\frac{s}{2}}\Delta_{V}(\tau-\Delta_{V})^{-2}d\tau  \label{2.20} \\
&&+c(s)\int_{0}^{\infty}
\tau^{\frac{s}{2}}(\tau-\Delta_{V})^{-1}W(\tau-\Delta_{V})^{-1}d\tau . \notag
\end{eqnarray}
Since
\begin{eqnarray}
s\left(-\Delta_{V}\right)^{\frac{s}{2}}=-2c(s)\int_{0}^{\infty}
\tau^{\frac{s}{2}}\Delta_{V}(\tau-\Delta_{V})^{-2}d\tau \notag
\end{eqnarray}
by integrating by parts,
we have our desired result from (\ref{2.20}).

\section{Appendix II: zero is not resonance}
\label{resonance}

In this section we can prove the lack of resonance at the origine, i.e. we shall prove that the origin is not resonance point, recalling that the definition of resonance used in \cite{JN01} and Theorem 6.2 guarantee that
zero is a resonance point can be characterized by the existence of solution
$$ \Psi(x) = c_0 + \Psi_0(x), \ \ c_0 \in \mathbb{C},\ \ \Psi_0 \in L^q(\mathbb{R}^2), \ \exists q \in (2,\infty)$$
to the equation
\begin{equation}\label{eq.A3.1}
    -\Delta \Psi + V \Psi = 0.
\end{equation}

Using  Lemma 6.4 and the relation (6.94) in \cite{JN01}, assuming $\beta > 10$, we can deduce further
$$ \Psi_0(x) = O(\langle x \rangle^{-1} ), \ \ $$
$$ \nabla \Psi_0(x) = O(\langle x \rangle^{-2} ).$$

Rewriting \eqref{eq.A3.1} in the form
$$ -\Delta \Psi_0 + V \Psi = 0, $$
multiplying by $\overline{\Psi}$ and integrating over $|x| \leq R,$ we get
$$ \int_{|x| \leq R} |\nabla \Psi_0(x)|^2 dx - \overline{c_0}\int_{|x|=R} \partial_r \Psi_0 (x) dS_x + \int_{|x| \leq R} V(x) | \Psi(x)|^2 dx =0.$$

The asymptotics of $\Psi_0, \partial_r \Psi_0$ enables one to take the limit $R \to \infty$ and arrive at
$$  \int_{\mathbb{R}^2} |\nabla \Psi_0(x)|^2 dx  + \int_{\mathbb{R}^2} V(x) | \Psi(x)|^2 dx =0$$
and the assumption $V \geq 0$ implies $\Psi=0.$


\bigskip
\textbf{Acknowledgments.}

V. Georgiev was supported in part by  Project 2017 ``Problemi stazionari e di evoluzione nelle equazioni di campo nonlineari" of INDAM,
GNAMPA - Gruppo Nazionale per l'Analisi Matematica,
la Probabilit\`a e le loro Applicazioni,
by Institute of Mathematics and Informatics,
Bulgarian Academy of Sciences and Top Global University Project, Waseda University,  by the University of Pisa, Project PRA 2018 49 and project ``Dinamica di equazioni nonlineari dispersive", ``Fondazione di Sardegna", 2016. C. Li was partially supported by the Education Department of Jilin Province [2018] and NNSFC under Grant Number 11461074.


\begin{thebibliography}{60}
\bibitem{Agmon} S. Agmon, \textit{Spectral, Properties of Schr\"{o}dinger operators and scattering theory}, Annali della Scuola Normale Superiore di Pisa-Classe di Scienze 2.2 (1975), 151--218.

\bibitem{Agmon2} S. Agmon, \textit{Lower bounds for solutions of Schr\"{o}dinger equations}, J. Analyse Math., \textbf{23} (1970), 1 -- 25.

\bibitem{AF2008} P. D'Ancona and L. Fanelli, \textit{Strichartz and smoothing estimates for dispersive equations with magnetic potentials}, Communications in Partial Differential Equations, \textbf{33} (2008), 1082--1112.

\bibitem{AFVV} P. D'Ancona, L. Fanelli, L. Vega, and N. Visciglia, \textit{Endpoint Strichartz estimates for the magnetic Schr\"{o}dinger equation}, Journal of Functional Analysis, \textbf{258} (2010), 3227--3240.

\bibitem{AM1998} P. Auscher, A. McIntosh, \textit{Heat kernels of second order complex elliptic operators and applications}, Journal of Functional Analysis, \textbf{152} (1998), 22--73.

\bibitem{Ba1984}  J. E. Barab, \textit{Nonexistence of asymptotically free
solutions for a nonlinear Schr\"{o}dinger equation}, J. Math. Phy., \textbf{25} (1984), 3270--3273.

\bibitem{BDS2016} K. Bogdan, J. Dziuba$\acute{n}$ski, and K. Szczypkowski,
\textit{Sharp Gaussian estimates for Schr\"{o}dinger heat kernels: L$^{p}$ integrability conditions}, 2016, arXiv: 1511.07167v3.

\bibitem{BM16}  J.-M. Bouclet and  H. Mizutani, \textit{Uniform resolvent and Strichartz estimates for Schr\"odinger equations with critical singularities,} Transactions of the American Mathematical Society, \textbf{370} (2018), 7293--7333.

\bibitem{BPSTZ2003} N. Burq, F. Planchon,  J. G. Stalker and A. S. Tahvildar-Zadeh, \textit{Strichartz estimates for the wave and Schr\"{o}dinger equations with the inverse-square potential}, Journal of Functional Analysis, \textbf{203} (2003), 519--549.

\bibitem{BPSTZ2004} N. Burq, F. Planchon,  J. G. Stalker and A. S. Tahvildar-Zadeh, \textit{Strichartz estimates for the wave and Schr\"{o}dinger equations with potentials of critical decay}, Indiana Univ. Math. J., \textbf{53} (2004), no. 6, 1665--1680.

\bibitem{CGV2014}  S. Cuccagna, V. Georgiev and N. Visciglia,  \textit{Decay and scattering of small
solutions of pure power NLS in $\mathbb{R}$ with $p>3$  and with  a potential,}
Communications on Pure and Applied Mathematics, \textbf{67} (2014), no.2, 957--981.

\bibitem{GI2005} V. Georgiev and A. Ivanov, \textit{Existence  and mapping properties of wave operator for the Schr\"{o}dinger  equation with singular potential},  Proceedings of the American Mathematical Society, \textbf{133} (2005),1993--2003.

\bibitem{GV2012}  V. Georgiev and B. Velichkov, \textit{Decay estimates for the supercritical 3-D Schr\"{o}dinger  equation with rapidly decreasing potential}, Progress in Mathematics, \textbf{301} (2012), 145--162.

  \bibitem{GV2007}
    V.  Georgiev and N. Visciglia, \textit{About resonances for Schr\"odinger operators with short range singular perturbation. Topics in contemporary differential geometry,} complex analysis and mathematical physics, World Sci. Publ., Hackensack, NJ, (2007), 74 -- 84.

 \bibitem{GOV} J. Ginibre, T. Ozawa and G. Velo, \textit{On the existence of wave operators for a class of nonlinear Schr\"{o}dinger equations,} Ann. Inst. H. Poincar$\rm{\acute{e}}$ Phys. Th$\rm{\acute{e}}$or.,  \textbf{60} (1994), 211--239.


\bibitem{HLN16} N. Hayashi, C. Li, and P. I. Naumkin, \textit{Nonlinear Schr\"{o}dinger systems in 2d with nondecaying final data}, Journal of Differential Equations, \textbf{260} (2016), Issue 2, 1472--1495.

\bibitem{HLN18} N. Hayashi, C. Li, and P. I. Naumkin, \textit{Critical nonlinear Schr\"{o}dinger equations in higher space dimensions}, J. Math. Soc. Japan, \textbf{70} (2018), no. 4, 1475--1492.

\bibitem{HN98}  N. Hayashi and P. Naumkin,  \textit{Asymptotics for large time of solutions to the nonlinear Schr\"{o} dinger and Hartree equations}, Amer. J. Math., \textbf{120} (1998), no. 2, 369--389.

\bibitem{HO1988} N. Hayashi and T. Ozawa, \textit{Scattering theory in the weighted $L^{2}(\mathbb{R}^{n})$ spaces for some Schr\"{o}dinger equations,} Ann. Inst. H. Poincare Phys. Theor., \textbf{48} (1988), 17--37.

\bibitem{JN01} A. Jensen and G. Nenciu, \textit{ A unified approach to resolvent expansions at thresholds.} Rev. Math. Phys. \textbf{13} (2001), no. 6, 717 -- 754.

\bibitem{JJL}
G. Jin, Y. Jin and C. Li, \textit{The initial value problem for nonlinear Schr\"{o}dinger equations with a dissipative nonlinearity in one space dimension}, Journal of Evolution Equations, \textbf{16} (2016), no.4, 983--995.

\bibitem{KeTao98} M. Keel and T. Tao, \textit{Endpoint Strichartz estimates,} Amer. J. Math.,
\textbf{2015} (1998), 955--980.

\bibitem{KS} N. Kita and A. Shimomura, \textit{Large time behavior of solutions to Schr\"{o}dinger equations with a dissipative nonlinearity for arbitrarily large initial data,} J. Math. Soc. Japan,
\textbf{61} (2009), no. 1, 39--64.

\bibitem{LS16} C. Li and H. Sunagawa, \textit{On Schr\"{o}dinger systems with cubic dissipative nonlinearities of derivative type,} \textbf{29}, (2016), no. 5, Nonlinearity,1537--1563.

\bibitem{Li2017} Z. Li and L. Zhao, \textit{Decay and scattering of solutions to nonlinear Schr\"{o}dinger equations with regular potentials for nonlinearities of sharp growth,} J. Math. Study, \textbf{50} (2017), 277--290.

\bibitem{MS91} H. P. McKean and J. Shatah, \textit{The nonlinear Schr\"{o}dinger equation and the nonlinear heat equation reduction to linear form,} Comm. Pure Appl. Math., \textbf{44} (1991), no. 8-9, 1067--1080.

\bibitem{Mi07} T. Mizumachi, \textit{Asymptotic stability of small solitons for 2D nonlinear Schr\"{o}dinger equations with potential,} J. Math. Kyoto Univ. (JMKYAZ), \textbf{47} (2007), 599--620.

\bibitem{Mo18}  K. Mochizuki, \textit{Spectral and scattering theory for second-order partial differential operators,} Monographs and Research Notes in Mathematics. CRC Press, Boca Raton, FL, 2017.

\bibitem{Ozawa} T. Ozawa, \textit{ Long range scattering for nonlinear Schr\"odinger equations in one space dimension,} Comm. Math. Phys., \textbf{139} (1991), no. 3, 479--493.


\bibitem{SaSuYa} Y. Sagawa, H. Sunagawa, and S. Yasuda, \textit{A sharp lower bound
for the lifespan of small solutions to the Schr\"{o}dinger equation with a subcritical power nonlinearity,}
Differential and Integral Equations, \textbf{31} (2018), no. 9-10, 685--700.

\bibitem{Sch} W. Schlag, \textit{Dispersive estimates for Schr\"{o}dinger operators in dimension two,} Comm. Math. Phys., \textbf{257} (2005), no.1, 87--117.

\bibitem{Si05} B. Simon, \textit{ Schr\"odinger semigroups,} Bull. Amer. Math. Soc. (N.S.) \textbf{7} (1982), no. 3, 447-- 526.

\bibitem{simon17} B. Simon, \textit{Tosio Kato's work on non-relativistic quantum mechanics: an outline}, (2017), arXiv: 1710.06999v1.

\bibitem{ste} A. Stefanov, \textit{Strichartz estimates for the magnetic Schr\"odinger equation,} Advances in Mathematics, \textbf{210} (2007), 246--303.

\bibitem{str74} W. Strauss, \textit{Nonlinear scattering theory. Scattering theory in mathematical physics, } In: Lavita J.A., Marchand JP. (eds) Scattering Theory in Mathematical Physics, NATO Advanced Study Institutes Series (Series C--Mathematical and Physical Sciences), vol 9. Springer, Dordrecht, 1974.


\bibitem{Str81} W. Strauss, \textit{Nonlinear scattering theory at low energy: sequel,} J. Funct. Anal. , \textbf{43} (1981), no. 3, 281--293.

\bibitem{Ya99} K. Yajima, \textit{$L^p$-boundedness of wave operators for two dimensional Schr\"{o}dinger operators}, Commun. Math. Phys., \textbf{208} (1999), 125--152.

\bibitem{Zhang2001} Q. Zhang, \textit{Global bounds of Schr\"{o}dinger heat kernels with negative potentials}, Journal of Functional Analysis, \textbf{182} (2001), 344--370.


\end{thebibliography}
\end{document}